\documentclass[a4paper,reqno,3p]{elsarticle}

\usepackage{graphicx} 
\usepackage{tikz}
\usetikzlibrary{patterns,shapes,decorations.pathmorphing,decorations.pathreplacing,calc,arrows,cd}
\usepackage{mathdots}
\usepackage{amssymb}
\usepackage{amstext}
\usepackage{amsmath}
\usepackage{amscd}
\usepackage{amsthm}
\usepackage{amsfonts}

\usepackage{array}
\usepackage{enumerate}
\usepackage{latexsym}
\usepackage{mathrsfs}
\usepackage[all,color]{xy}
\usepackage{mathtools}
\usepackage{pgfplots}
\usepackage{booktabs}
\usepackage{multirow}
\xyoption{all}
\usepackage{pstricks}
\usepackage{comment}

\usepackage{todonotes}

\date{\today}

\numberwithin{equation}{section}
\newtheorem{theorem}{Theorem}[section]

\newtheorem{corollary}[theorem]{Corollary}
\newtheorem{lemma}[theorem]{Lemma}
\newtheorem{proposition}[theorem]{Proposition}
\newtheorem{definition-theorem}[theorem]{Definition-Theorem}
\newtheorem{definition-proposition}[theorem]{Definition-Proposition}

\theoremstyle{definition}
\newtheorem{definition}[theorem]{Definition}
\newtheorem{assumption}[theorem]{Assumption}
\newtheorem{remark}[theorem]{Remark}
\newtheorem{example}[theorem]{Example}

\newtheorem{condition}[theorem]{Condition}

\renewcommand{\P}{\mathit{P}}

\newcommand{\soc}{\operatorname{soc}\nolimits}

\newcommand{\Hom}{\operatorname{Hom}\nolimits}

\newcommand{\RHom}{\mathbf{R}\strut\kern-.2em\operatorname{Hom}\nolimits}

\newcommand{\Image}{\operatorname{Im}\nolimits}

\DeclareMathOperator{\moduleCategory}{\mathsf{mod}} \renewcommand{\mod}{\moduleCategory}

\DeclareMathOperator{\ind}{\mathsf{ind}}

\DeclareMathOperator{\add}{\mathsf{add}}
\DeclareMathOperator{\Fac}{\mathsf{Fac}}

\newcommand{\Hasse}{\operatorname{Hasse}\nolimits}

\DeclareMathOperator{\resdim}{\!\mbox{-}\mathrm{res}\mbox{-}\mathrm{dim}}
\DeclareMathOperator{\supp}{\mathrm{supp}}

\DeclareMathOperator{\rep}{\mathsf{rep}}
\DeclareMathOperator{\intresdim}{\mathrm{int}\mbox{-}\mathrm{res}\mbox{-}\mathrm{dim}}
\DeclareMathOperator{\intresgldim}{\mathrm{int}\mbox{-}\mathrm{res}\mbox{-}\mathrm{gldim}}
\DeclareMathOperator{\resgldim}{\!\mbox{-}\mathrm{res}\mbox{-}\mathrm{gldim}}

\newcommand{\conv}{\mathrm{conv}}

\DeclareMathOperator{\Ker}{\mathrm{Ker}}

\newcommand{\res}{\mathrm{Res}}
\newcommand{\iind}{\mathrm{Ind}}
\newcommand{\coind}{\mathrm{Coind}}
\newcommand{\trace}{\mathrm{Trace}}

\pgfplotsset{compat=1.16}

\begin{document}

\begin{frontmatter}
\title{Summand-injectivity of interval covers and \\ monotonicity of interval resolution global dimensions}

\author[1]{Toshitaka Aoki}
\ead{toshitaka.aoki@people.kobe-u.ac.jp}
 
\author[1]{Emerson G.\ Escolar\corref{cor1}}
\cortext[cor1]{Corresponding author}
\ead{e.g.escolar@people.kobe-u.ac.jp}

\author[1]{Shunsuke Tada} 
\ead{205d851d@stu.kobe-u.ac.jp}

\affiliation[1]{organization={Graduate School of Human Development and Environment, Kobe University}, 
  addressline={3-11 Tsurukabuto, Nada-ku},
  postcode={657-8501}, 
  city={Kobe}, 
  country={Japan}
  }


\begin{abstract} 
Recently, there is growing interest in the use of relative homological algebra to develop invariants using interval covers and interval resolutions (i.e., right minimal approximations and resolutions relative to interval-decomposable modules) for multi-parameter persistence modules.
In this paper, the set of all interval modules over a given poset plays a central role. 
Firstly, we show that the restriction of interval covers of modules to each indecomposable direct summand is injective. 
This result suggests a way to simplify the computation of interval covers.
Secondly, we show the monotonicity of the interval resolution global dimension, i.e., if $Q$ is a full subposet of $P$, then the interval resolution global dimension of $Q$ is not larger than that of $P$. 
Finally, we provide a complete classification of posets whose interval resolution global dimension is zero. 
\end{abstract}

\end{frontmatter}

\section{Introduction}

Topological data analysis is a rapidly growing field applying the ideas of algebraic topology for data analysis.
One of its main tools is persistent homology \cite{frosini1999size,landi1997new,robins1999towards,carlsson2005computing}, which can compactly summarize the birth and death parameters (persistence intervals) of topological features (e.g.\ connected components, rings, cavities, and so on) of data via the persistence diagram.
This allows us to analyze hidden structures in data.
For example, in the field of material science, the analysis unveiled a hierarchical ring structure in silica glass \cite{hiraoka2016hierarchical}.
There are a lot of other applications,
including in the field of evolutionary biology  \cite{evolution},
cosmic web \cite{sousbie1_2011persistent,sousbie2_2011persistent},
and so on \cite{mcguirl2020topological,Jaquette_2020,aktas2019persistence,belchi2018lung}.   
Interested readers may refer to \cite{carlsson_vejdemo-johansson_2021,rabadan_blumberg_2019} for more details.

Algebraically, one part of the persistent homology analysis can be formalized
by using the so-called one-parameter persistence modules,
which are just (``pointwise'') finite dimensional modules over the incidence algebra of a totally ordered set.
In this point of view, one-parameter persistence modules are guaranteed to decompose into the \emph{interval modules} \cite{botnan2020decomposition}, which provides a multiset of intervals, called the persistence diagram.

As a generalization, multi-parameter persistence modules have been proposed \cite{carlsson2009theory},
which are understood as representations of $n$-dimensional grids,
and are expected to provide richer information compared to the one-parameter setting.
However,
unlike one-parameter persistence modules, there is no complete discrete invariant to capture all the
indecomposable modules \cite{carlsson2009theory}.
Furthermore, when dealing with large grids, the incidence algebra is known to be of wild representation type \cite{leszczynski1994representation,leszczynski2000tame} (see also \cite{bauer2020cotorsion}).
Moreover, there are computational difficulties when working with multi-parameter persistence modules \cite{bjerkevik2020computing}.
Thus applying similar techniques as the one-parameter case is theoretically and computationally challenging.

Recently, there is an interest in the use of relative homological algebra
in persistence theory
\cite{Botnan2021SignedBF,blanchette2021homological,chacholski2023koszul,asashibakoszul2023,asashiba2023approximation,blanchette2023exact}.
Especially, the notion of interval covers and interval resolutions (i.e. right minimal approximations and resolutions by interval-decomposable modules respectively) are developed, and the finiteness of the interval resolution global dimension has been confirmed \cite{asashiba2023approximation}.

An aim of this paper is to study the properties of interval covers and interval resolutions by interval-decomposable modules. 
Firstly, we show the following result, which asserts that such an approximation is given by a family of interval submodules. 

\begin{theorem}[Corollary \ref{cor:interval cover}]\label{main result1}
    Let $P$ be a finite poset and $k[P]$ the incidence algebra of $P$. 
    For a given $k[P]$-module $M$, we take its interval cover
    $f\colon X=\bigoplus_{i=1}^m X_i \to M$ (where all the $X_i$'s are interval modules). Then, the following holds.
    \begin{enumerate}[\rm (1)]
        \item $f$ is surjective. 
        \item $f|_{X_i}\colon X_i\to M$ is injective for every $i\in \{1,\ldots,m\}$. 
        \item The supports of $M$ and $X$ coincide. 
    \end{enumerate}    
    In particular, every $X_i$ can be taken as an interval submodule of $M$.
\end{theorem}

A similar observation has been discussed in \cite[Proposition 6.7]{blanchette2021homological}.
We note that Theorem~\ref{main result1} is essentially the same as \cite[Proposition 4.8]{asashibakoszul2023}, which was shown independently. See Remark~\ref{rem:asashiba}.
The importance of Theorem \ref{main result1} is that it 
provides one way to reduce the computational burden for computing interval resolutions:
instead of the original poset $P$ one can compute over smaller and smaller posets starting from the support of $M$.
We provide details
in Example \ref{Example:resolutiondimension}.

Secondly, we study the relationship between the interval resolution global dimensions 
of different posets.
We show the following.
\begin{theorem}[Theorem \ref{thm:monotone subposet}] \label{main result2}
    Let $P$ be a finite poset and $k[P]$ the incidence algebra of $P$. 
    For any full subposet $Q$ of $P$, the interval resolution global dimension of $k[Q]$ is smaller than or equal to that of $k[P]$.
\end{theorem}
The main ingredient for its proof is the functor $\Theta \colon \mod k[Q] \to \mod k[P]$ 
called the \emph{intermediate extension} \cite{kuhn1994generic} (the \emph{prolongement interm\'ediare} in \cite{BBD}) between the module categories, which is defined by using idempotent embeddings (see Section \ref{sec:idempotent}).
We prove that $\Theta$ realizes a combinatorial operation called the \emph{convex hull}
in the level of module categories (Lemma \ref{Theta->int}). 
As a consequence, we obtain a pair of functors 
\begin{equation*}
    \xymatrix{
    \mod k[P] \ar@{->}@/_3mm/[rr]^{\res} & & \mod k[Q] \ar@{->}@/_3mm/[ll]_{\rm \Theta}. 
    }
\end{equation*}
such that $\res$ and $\Theta$ preserve interval-decomposability of modules and satisfy $\res\circ \Theta\cong 1_{\mod k[Q]}$. 
Then, we can directly compare interval resolutions via these functors and obtain the
monotonicity of interval resolution global dimension with respect to (full subposet) inclusion.
It is interesting because it does not hold for the usual global dimension in general \cite{igusa1990cohomology}.
We also note that it does not hold if we drop the condition that $Q$ be a \emph{full} subposet of $P$ (Remark~\ref{rem:nonfull}).

Finally, we give a complete classification of finite posets whose interval resolution global dimension is zero.
It generalizes the setting of one-parameter persistence modules. 

\begin{theorem}[Theorem \ref{thm:gldim=0}] \label{main result3}
Let $P$ be a finite poset and $k[P]$ the incidence algebra of $P$. 
Then, the following conditions are equivalent. 
\begin{enumerate}[\rm (a)]
    \item The interval resolution global dimension of $P$ is zero (equivalently, all modules are interval-decomposable). 
    \item Each connected component of the Hasse diagram of $P$ is one of $A_n(a)$ for some  orientation $a$ or $C_{m,\ell}$ displayed below, where the symbol $\leftrightarrow$ is either $\rightarrow$ or $\leftarrow$ assigned by its orientation $a$: 
    \begin{eqnarray*}
    &A_n(a)\colon& \quad \quad \quad  1 \longleftrightarrow 2  \longleftrightarrow \cdots \cdots \longleftrightarrow n, \\ \\ 
    &C_{m,\ell}\colon&\begin{tikzpicture}[baseline=0mm]
        \coordinate (x) at (1,0);
        \coordinate (y) at (0,0.6);
        \node (0h) at (0,0) {$\hat{0}$};
        \node (1h) at (6,0) {$\hat{1}$};
        \node (1) at ($1*(x) + (y)$) {$1$};
        \node (2) at ($2*(x) + (y)$) {$2$};
        \node (3) at ($3*(x) + (y)$) {$3$};
        \node (m) at ($5*(x) + (y)$) {$m$};
        \node (1p) at ($1*(x) + -1*(y)$) {$1'$};
        \node (2p) at ($2*(x) + -1*(y)$) {$2'$};
        \node (3p) at ($3*(x) + -1*(y)$) {$3'$};
        \node (l) at ($5*(x) + -1*(y)$) {$\ell'$};
        \draw[->] (0h)--(1);
        \draw[->] (0h)--(1p);
        \draw[->] (m)--(1h);
        \draw[->] (l)--(1h);
        
        \draw[->] (1)--(2);
        \draw[->] (2)--(3);
        \draw[->] (3)--(m);
        \node[fill=white] at ($(3)!0.5!(m)$) {$\cdots$}; 
        \draw[->] (1p)--(2p);
        \draw[->] (2p)--(3p);
        \draw[->] (3p)--(l);
        \node[fill=white] at ($(3p)!0.5!(l)$) {$\cdots$}; 
    \end{tikzpicture}
\end{eqnarray*} 
\end{enumerate}
In particular, these conditions do not depend on the characteristic of the base field $k$.
\end{theorem}

To prove this, we use knowledge of special biserial algebras \cite{WW,SW,BR87}.

This paper is organized as follows. In Section \ref{sec2}, we recall basic terminology of relative homological algebra. 
We also discuss interval modules and interval resolution dimensions over incidence algebras of finite posets.
In Section \ref{sec3}, we study approximations and resolutions relative to full subcategories which are closed under quotients  of indecomposable modules and prove Theorem \ref{thm:resolusion-quotient}.
We will apply this result to the class of interval modules and prove Theorem \ref{main result1} (Section \ref{sec:convex}).
In Section \ref{sec4}, we study interval resolution dimensions for full subposets. In Section \ref{sec:idempotent}, we define the functor $\Theta$ by using idempotent embedding functors. Using this, we prove Theorem \ref{main result2}.
In Section \ref{sec5}, we prove Theorem \ref{main result3}.


\section{Preliminaries}\label{sec2}
In this section, we recall the basics of the representation theory of finite dimensional algebras.
We refer to \cite{ASS,ARS} for definitions and fundamental results. 
Let $A$ be a finite dimensional algebra over a field $k$. We denote by $\mod A$ the category of finitely generated right $A$-modules. 
Throughout this paper, we assume that all modules are finitely generated. For morphisms $f\colon X\to Y$ and $g\colon Y\to Z$ of $A$-modules, we denote their composition by $gf\colon X\to Z$.

\subsection{Approximations and resolutions}
We recall the basic terminology of relative homological algebra.
We follow the works \cite{blanchette2021homological,asashiba2023approximation}
also applying relative homological algebra in persistence theory, but
change the notation slightly for our purposes.

\begin{definition}
Let $\mathcal{X}$ be a full subcategory of $\mod A$. 
For a morphism $f\colon X\to M$ of $A$-modules, we say that
\begin{enumerate}[\rm (1)]
    \item $f$ is \emph{right minimal} if any morphism $g: X \to X$ satisfying $fg=f$ is an isomorphism.
    \item $f$ is a \emph{right $\mathcal{X}$-approximation} of $M$
    if $X \in \mathcal{X}$ and $\Hom_{A} (Y, f)$ is surjective for any $Y \in \mathcal{X}$.
    \item $f$ is a \emph{right minimal $\mathcal{X}$-approximation} of $M$ if it is a right $\mathcal{X}$-approximation which is right minimal.
\end{enumerate} 
\end{definition}

In this paper, we mainly consider approximations 
by the full subcategory $\mathcal{X} := \add \mathscr{X}$ for a fixed finite collection $\mathscr{X}$ of (isomorphism classes of) indecomposable $A$-modules.
Here, $\add \mathscr{X}$ denotes the smallest full subcategory of $\mod A$ which contains $\mathscr{X}$ and is closed under taking isomorphisms, direct sums and direct summands. 
In this situation, every $A$-module $M$ admits a right minimal $(\add \mathscr{X})$-approximation which is uniquely determined up to isomorphism.
Furthermore, it is surjective when $\mathscr{X}$ contains all indecomposable projective $A$-modules. 

Now, we assume that $\mathscr{X}$ contains all indecomposable projective $A$-modules. 
For an $A$-module $M$, we take a right minimal $(\add \mathscr{X})$-approximation $f\colon J\to M$. 
Then, it induces a short exact sequence
\begin{equation}\nonumber
    0 \longrightarrow \Ker(f) \overset{\iota}{\longrightarrow} J \overset{f}{\longrightarrow} M \longrightarrow 0, 
\end{equation}
where $\iota$ is the kernel of $f$. 
We set $\Omega_{\mathscr{X}}(M) :=\Ker(f)$ and call it \emph{$\mathscr{X}$-syzygy} of $M$. 
Notice that it is uniquely determined up to isomorphism.
In addition, we set $\Omega_{\mathscr{X}}^0(M):=M$ and $\Omega^i_{\mathscr{X}}(M):=\Omega_{\mathscr{X}}(\Omega_{\mathscr{X}}^{i-1}(M))$ for all $i>0$. Under these notations, we obtain a diagram 
\begin{equation}\label{resolution of M}
        \xymatrix@C=15pt{&\ar@{.}[r] & \ar[r]& J_2 \ar[rr]^{g_2}\ar@{.>}[rd]_{f_2}&&  J_1\ar[rr]^{g_1}\ar@{.>}[rd]_{f_1}&& J \ar[rr]^{f} && M \ar[rr] && 0 \\
        &&&&\ar@{}[u]|{\circlearrowright}\Omega_{\mathscr{X}}^2(M)\ar@{.>}[ur]_{\iota_1} && \ar@{}[u]|{\circlearrowright}\Omega_{\mathscr{X}}(M)\ar@{.>}[ur]_{\iota}}
\end{equation}
where $f_i\colon J_i \to \Omega_{\mathscr{X}}^i(M)$ is a right minimal $(\add \mathscr{X})$-approximation with $J_i\in \add \mathscr{X}$ for every $i>0$, and $g_i$ is the composition of $f_i$ and the kernel $\iota_{i-1}$ of $f_{i-1}$ for every $i>0$, where $\iota_0:=\iota$.
An exact sequence \eqref{resolution of M} is called a \emph{right minimal $\mathscr{X}$-resolution} of $M$, which is uniquely determined up to isomorphism. 

\begin{definition}
If $M$ has a right minimal $\mathscr{X}$-resolution of the form
\begin{equation}\nonumber
    0 \longrightarrow J_m \overset{g_m}{\longrightarrow}
    \cdots \overset{g_2}{\longrightarrow} J_1\overset{g_1}{\longrightarrow} J_0 \overset{f}{\longrightarrow} M \longrightarrow 0,
\end{equation}
then we say that the \emph{$\mathscr{X}$-resolution dimension} of $M$ is $m$ and write $\mathscr{X}\resdim M = m$.
Otherwise, we say that the $\mathscr{X}$-resolution dimension of $M$ is infinity.

We set
\begin{equation}
    \mathscr{X}\resgldim A := \sup \{\mathscr{X}\resdim M \mid M \in \mod A\}
\end{equation}
and call \emph{$\mathscr{X}$-resolution global dimension} of $A$. Notice that it can be infinity. 
\end{definition}

\begin{remark}
    We define the resolution dimension by using right minimal approximations. 
    Equivalently, the resolution dimension is defined as the infimum of the length of (not necessarily minimal) resolutions of a given module, see \cite[Proposition 3.9]{asashiba2023approximation}. 
\end{remark}

\subsection{Partially ordered sets}
\label{subsec:poset}
A \emph{poset} is a set $P$ equipped with a partial order $\leq$.
Let $P$ be a finite poset.
We denote by $\Hasse(P)$ the Hasse diagram of $P$, that is,
the set of vertices is $P$ and we draw
an arrow $x \to y$ for $x,y\in P$ if and only if $x < y$ and
there is no $z\in P$ such that $x < z < y$. In particular, $\Hasse(P)$ is acyclic.
The \emph{incidence algebra} $k[P]$ of $P$ is defined to be
the quotient of the path algebra of $\Hasse(P)$ over $k$
modulo the two-sided ideal generated by all the commutative relations.
In the rest of this paper, we assume that $P$ is \emph{connected} (i.e., $\Hasse(P)$ is a connected quiver) unless otherwise specified. In this case, the corresponding incidence algebra $k[P]$ is connected.

Let $P$ be a finite poset and $k[P]$ be the incidence algebra of $P$.
Then, the module category $\mod k[P]$ can be described in terms of a functor category as follows.
Firstly, we regard $P$ as a category whose objects are elements of $P$, and morphisms are defined by relations in $P$, i.e., there is a unique morphism $a \to b$ for $a,b\in P$ if and only if $a\leq b$. We denote by ${\sf rep}_k(P)$ the category of (covariant) functors from $P$ to the category of finite dimensional vector spaces over $k$. 
More explicitly, objects and morphisms of this category are given as follows: An object $V$ of ${\sf rep}_k(P)$ is a correspondence such that
\begin{enumerate}
    \item it associates an element $a\in P$ to a finite dimensional $k$-vector space $V_a$, and
    \item it associates a relation $a\leq b$ in $P$ to a $k$-linear map $V(a\leq b)$ from $V_a$ to $V_b$ in such a way that $V(a\leq a) = {\rm 1}_{V_a}$ and $V(b\leq c)\circ V(a\leq b) = V(a\leq c)$ for all $a,b,c\in P$.
\end{enumerate}
For $V$ in ${\sf rep}_k(P)$, the subset
\begin{equation}\label{suppV}
    \supp V := \{ a\in P \mid V_a \neq 0 \}
\end{equation}
is called the \emph{support} of $V$. For two objects $V,W\in {\sf rep}_k(P)$, a morphism from $V$ to $W$ is a family of $k$-linear maps $f_a : V_a \to W_a$ with $a\in P$ satisfying $f_b \circ V(a \leq b) = W(a \leq b) \circ f_a$ for every relation $a \leq b$ in $\P$, that is, the following diagram is commutative.
\[
\xymatrix{
    V_a \ar[rr]^-{V(a \leq b)} \ar[d]^{f_a}
    & &V_b \ar[d] ^{f_b}
    \\
    W_a \ar[rr]^-{W(a \leq b)}
    & \ar@{}[u]|{\circlearrowright} &W_b
}
\]

It is well-known that there is an equivalence of abelian categories between ${\sf rep}_k(P)$ and the module category $\mod k[\P]$ of the incidence algebra of $\P$. 
In this sense, we identify objects $V$ of $\rep_k(P)$ and $k[\P]$-modules. From \eqref{suppV}, the \emph{support} of a $k[P]$-module $M$ is the subset 
\begin{equation}
    \supp(M)= \{a\in P \mid Me_a\neq 0 \},
\end{equation}
where $e_a$ is a primitive idempotent of $k[P]$ corresponding to the element $a\in P$.

\begin{definition}\label{def:conv}
    A \emph{full subposet} of $\P$ is a subset $P' \subseteq P$ equipped with the induced partial order. Notice that it is completely determined by its elements. 
    We say that 
    \begin{enumerate}[\rm (1)]
        \item $P'$ is \emph{convex} in $P$ if, for any $x,y\in P'$ and any $z\in P$, $x < z < y$ implies $z\in P'$,
        \item $P'$ is an \emph{interval} of $P$ if $P'$ is connected as a poset and is convex in $P$.
    \end{enumerate}
We denote by $\mathbb{I}(P)$ the set of intervals of $P$.
\end{definition}

The following special class of modules plays an important role in this paper.

\begin{definition}\label{def:intervalmod}
    For an interval $I$ of $P$, let $k_I$ be a $k[P]$-module given as follows.
    \begin{equation}
        (k_I)_a =
        \begin{cases}
        k & \text{if $a\in I$}, \\
        0 & \text{otherwise,}
        \end{cases}\quad \quad
        k_I(a\leq b) =
        \begin{cases}
        1_k & \text{if $a,b\in I$}, \\
        0 & \text{otherwise.}
        \end{cases}
    \end{equation}
    An \emph{interval module} is a $k[P]$-module $M$ such that $M\cong k_I$ for some interval $I\in \mathbb{I}(P)$. 
    Clearly, every interval module is indecomposable. 
\end{definition}

We denote by $\mathscr{I}_{k, P}$ the set of isomorphism classes of interval $k[P]$-modules, which is in bijection with $\mathbb{I}(P)$ by $I\mapsto k_I$. Notice that $\mathbb{I}_{P}$ and $\mathscr{I}_{P,k}$ are finite since so is $P$. 
Each module in $\add \mathscr{I}_{P,k}$ is said to be \emph{interval-decomposable}.
In other words, a given $k[P]$-module $M$ is interval-decomposable if and only if it can be written as 
\begin{equation}\nonumber
    M\cong \bigoplus_{I\in \mathbb{I}(P)} k_I^{m(I)}
\end{equation}
for some non-negative integers $m(I)$. 
We will write $\mathscr{I}_P$ instead of $\mathscr{I}_{k,P}$ when the base field $k$ is clear.

Since $\mathscr{I}_{P}$ contains all indecomposable projective $k[P]$-modules by definition, one can consider resolutions by interval modules. 
By \emph{interval covers over $P$} (resp., \emph{interval resolutions over $P$}), we mean
right minimal $(\add \mathscr{I}_P)$-approximations (resp., $\mathscr{I}_P$-resolutions) of $k[P]$-modules. Where the poset $P$ is clear, we may omit it.
In addition, we will write 
\begin{equation}\nonumber
    \intresdim M := \mathscr{I}_P\resdim M \quad \text{and}\quad 
    \intresgldim k[P] := \mathscr{I}_P\resgldim k[P], 
\end{equation}
and call them the \emph{interval resolution dimension} of a module $M$ and the \emph{interval resolution global dimension} of $k[P]$ respectively.
It has been shown in \cite[Proposition 4.5]{asashiba2023approximation} that the interval resolution global dimension is always finite (see also Theorem \ref{intgldim-finite}).
By definition, $k[P]$ has interval resolution global dimension zero if and only if every $k[P]$-module is interval-decomposable, if and only if every indecomposable $k[P]$-module is interval. In Section \ref{Section:Dim=0}, we will give a complete classification of such posets.

Before giving examples, we fix some notations.
Let $G$ be a simple undirected finite graph.
The directed graph whose underlying undirected graph is $G$ and has 
orientation\footnote{
Precisely speaking, an orientation can be formalized as a function that assigns to each undirected edge
$\{u,v\}$ in $G$ a direction $(u,v)$ or $(v,u)$. We do not need this 
formalism.}
$a$ will be denoted by $G(a)$.
If $G(a)$ is acyclic and has no arrows of the form
\begin{tikzpicture}[baseline=0mm]
    \node (a) at (0,0) {};
    \node (b) at (1,0) {};
    \coordinate (c) at (0.5,0.3) ;
    \draw[->] (a)--($(c)+(-0.05,0)$);
    \draw[->] (a)--(b);
    \draw[->] ($(c)+(0.05,0)$)--(b);
\end{tikzpicture}, then it is realized as the Hasse diagram
of a poset.
In fact, it is given by the poset whose elements are the vertices of $G$ and the partial order is generated by the arrows of $G$. 
In this way, we sometimes identify such a directed graph with a poset. 

For instance, we consider the Dynkin diagram of type $A_n$. Then, the $A_n$-type quiver $A_n(a)$ is of the form 
\begin{equation}
    1 \longleftrightarrow 2  \longleftrightarrow \cdots \cdots \longleftrightarrow n,
\end{equation}
where $\leftrightarrow$ is either $\rightarrow$ or $\leftarrow$ as assigned by the orientation $a$.
Important cases are
\begin{itemize}
    \item the \emph{equioriented} $A_n$-type quiver $A_n(e)$ given by
    \begin{equation}\nonumber
        1 \longrightarrow 2 \longrightarrow \cdots \cdots \longrightarrow n, 
    \end{equation}
  \item and a \emph{purely zigzag} $A_n$-type quiver $A_n(z)$ given by 
    the following (or its opposite):
    \begin{equation}\nonumber
        1 \longrightarrow 2 \longleftarrow \cdots \cdots \longleftrightarrow n
      \end{equation}
      where the last arrow is $\leftarrow$ (resp., $\rightarrow$) if $n$ is odd (resp., even).

\end{itemize} 
In addition, the $D_n$-type quiver $D_n(b)$ is of the form 
\begin{equation}\nonumber
    \begin{tikzpicture}
    \node (u) at (0,1) {$1$};
    \node (a) at(-1,0) {$2$}; 
    \node (c) at(0,0) {$3$};
        
    \node (b) at(1.5,0) {$\longleftrightarrow\cdots \cdots \longleftrightarrow n$};
    \draw[<->] (u)--(c); 
    \draw[<->] (a)--(c); 
    \end{tikzpicture} 
\end{equation}

\begin{example}\label{Example:Glintdim=1} 
Here, we compute interval resolution global dimension for a few examples. 
\begin{enumerate}[\rm (1)]
\item For any $A_n$-type quiver $A_n(a)$, Gabriel's theorem \cite{Gabriel72} asserts 
that every indecomposable module is an interval module. 
Therefore, the interval resolution global dimension for $A_n(a)$ is zero.  
\item Next, we consider the $D_4$-type quiver $D_4(b)$ displayed below: 
\begin{equation}\nonumber
    \begin{tikzpicture}
    \node (a) at (0,1) {$1$};
    \node (b) at(-1,0) {$2$}; 
    \node (c) at(0,0) {$3$};
    \node (d) at(1,0) {$4.$}; 
    \draw[->] (a)--(c); 
    \draw[->] (c)--(b); 
    \draw[->] (c)--(d);
    \end{tikzpicture} 
\end{equation}
Then, the incidence algebra is just a path algebra of type $D_4$.
The Auslander-Reiten quiver is given by
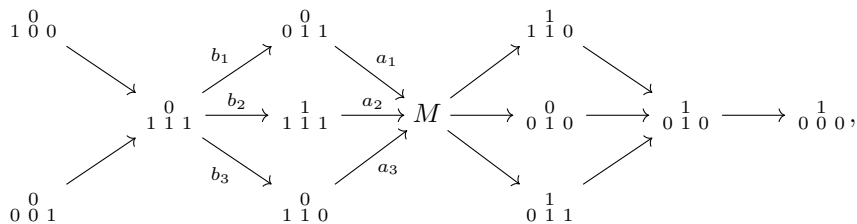
\[\begin{tikzcd}[ampersand replacement = \&, cramped, ]
 \begin{smallmatrix}
 & 0 &  \\
 1& 0 & 0  \\
\end{smallmatrix} \arrow[dr]   \&  \&  
\begin{smallmatrix}
 & 0 &  \\
 0& 1 & 1  \\
\end{smallmatrix} \arrow[rd,pos=0.5,"a_1"]
\& \&
\begin{smallmatrix}
 & 1 &  \\
 1& 1 & 0  \\
\end{smallmatrix} \arrow[rd]
\\
   \& \begin{smallmatrix}
 & 0 &  \\
 1& 1 & 1  \\
\end{smallmatrix} \arrow[r,pos=0.5,"b_2"]  \arrow[ru,pos=0.5,"b_1"] \arrow[rd,pos=0.5,"b_3"']  \& \begin{smallmatrix}
 & 1 &  \\
 1& 1 & 1   \\
\end{smallmatrix}  \arrow[r,pos=0.5,"a_2"] \& M \arrow[r]  \arrow[ru] \arrow[rd] \& \begin{smallmatrix}
 & 0 &  \\
 0& 1 & 0   \\
\end{smallmatrix} \arrow[r] \& \begin{smallmatrix}
 & 1 &  \\
 0 & 1 & 0  \\
\end{smallmatrix}  \arrow[r] \& \begin{smallmatrix}
 & 1 &  \\
 0 & 0 & 0  \\
\end{smallmatrix},
 \\
 \begin{smallmatrix}
 & 0 &  \\
 0& 0 & 1  \\
\end{smallmatrix} \arrow[ru]   \&   \& \begin{smallmatrix}
 & 0 &  \\
 1& 1 & 0  \\
\end{smallmatrix} \arrow[ru,pos=0.5,"a_3"'] \&  \&  \begin{smallmatrix}
 & 1 &  \\
 0& 1 & 1  \\
\end{smallmatrix}  \arrow[ru]
\end{tikzcd}
\]
where all indecomposable modules except for $M$ are interval, but $M$ is
\[
\xymatrix{  &  k\ar[d]^(0.3){{}^t[1\,1]} & \\ k  &
 \ar[l]_{[1\,0]}   k^2 \ar[r]^{[0\,1]}  & k.}
\]
Looking at the Auslander-Reiten quiver,
we find that a minimal interval resolution of $M$ is
\[ 0 \longrightarrow
\begin{smallmatrix}
 & 0 &  \\
1 & 1 & 1 \\
\end{smallmatrix}
\xrightarrow{{}^t[b_1,b_2,b_3]}
\begin{smallmatrix}
 & 0 &  \\
 0& 1 & 1  \\
\end{smallmatrix} \oplus
\begin{smallmatrix}
 & 1 &  \\
1 & 1 &  1 \\
\end{smallmatrix} \oplus
\begin{smallmatrix}
 & 0 & \\
 1& 1 &  0 \\
\end{smallmatrix}
\xrightarrow{[a_1,a_2,a_3]}
M
\longrightarrow 0, 
\]
and hence
\begin{equation}\nonumber
\intresdim M = 1.
\end{equation}
Consequently, the interval resolution global dimension for $D_4(b)$ is $1$
because $M$ is the only non-interval indecomposable module.
By a similar discussion, one can show that any $D_4$-type quiver has the interval resolution global dimension $1$.

\item The commutative ladders \cite{escolar2016persistence},
a special class of equioriented commutative 2D grids, can be thought of as a restricted setting for two-parameter persistence. 
A commutative ladder is a poset with $2m$ elements ($m>0$) and given by the following Hasse diagram:
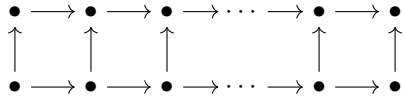
\begin{equation*}
\begin{tikzpicture}
    \node (10) at (1,0) {$\bullet$}; 
    \node (20) at (2,0) {$\bullet$}; 
    \node (30) at (3,0) {$\bullet$}; 
    \node (50) at (5,0) {$\bullet$}; 
    \node (11) at (1,1) {$\bullet$}; 
    \node (21) at (2,1) {$\bullet$}; 
    \node (31) at (3,1) {$\bullet$}; 
    \node (51) at (5,1) {$\bullet$};
    \node (61) at (6,1) {$\bullet$}; 
    \node (60) at (6,0) {$\bullet$};
    \node (3'0) at (3.8,0) {}; 
    \node (3'1) at (3.8,1) {};
    \node (5'0) at (4.2,0) {}; 
    \node (5'1) at (4.2,1) {};
    \node at ($(31)!0.5!(51)$) {$\cdots$};
    \node at ($(30)!0.5!(50)$) {$\cdots$};
    \draw[->] (10)--(20); 
    \draw[->] (20)--(30); 
    \draw[->] (10)--(11); 
    \draw[->] (20)--(21);
    \draw[->] (30)--(31);
    \draw[->] (11)--(21); 
    \draw[->] (21)--(31); 
    \draw[->] (50)--(51); 
    \draw[->] (60)--(61); 
    \draw[->] (30)--(3'0);
    \draw[->] (31)--(3'1);
    \draw[<-] (50)--(5'0);
    \draw[<-] (51)--(5'1);
    \draw[->] (50)--(60);
    \draw[->] (51)--(61);
\end{tikzpicture}
\end{equation*}

In representation theory, it is well-known (see \cite{leszczynski1994representation,leszczynski2000tame,bauer2020cotorsion}  for example) that the corresponding incidence algebra is of finite (resp., tame, wild) representation type if and only if $1\leq m\leq 4 $ (resp., $m=5$ and $6 \leq m$).
In \cite{asashiba2023approximation}, they study several numerical invariants associated with interval resolutions and compute interval resolution global dimensions for small $m$ (over a field with $2$ elements). 
\end{enumerate}
\end{example}


\section{Resolution by subcategories closed under quotients of indecomposable objects} \label{sec3}
In this section, we study approximations and resolutions by a certain class of subcategories. 
Throughout this section, let $A$ be a finite dimensional algebra over a field $k$.

\subsection{Resolutions and supports}
We begin with the following definition.

\begin{definition}\label{Def:indec-quo}
Let $\mathcal{X}$ be a full subcategory of $\mod A$. 
We say that $\mathcal{X}$ is 
closed under quotients (resp., submodules) of indecomposable modules if, for any short exact sequence
\begin{equation}
0\to Z\to X \to Y \to 0
\end{equation}
with $X$ indecomposable, $X\in \mathcal{X}$ implies $Y \in \mathcal{X}$ (resp., $Z\in \mathcal{X}$).
\end{definition}

Our definition is motivated by the next finiteness which  follows from {\cite[Theorem in \S 5]{ringel2010iyama}} (cf. \cite[Lemma~2.2]{iyama2003finiteness}),
but see also \cite[Corollary~4.3]{asashiba2023approximation}.

\begin{theorem}\label{resolution-finite}
Let $\mathscr{X}$ be a finite collection of indecomposable $A$-modules which contains all projective modules. 
If $\add \mathscr{X}$ is 
closed under submodules of indecomposable modules, then the $\mathscr{X}$-resolution global dimension is finite.  
\end{theorem}

In this case, the endomorphism algebra of the direct sum of all modules in $\mathscr{X}$ is known to be a \emph{left strongly quasi-hereditary algebra}
(see \cite{dlab_ringel_1992} for backgrounds of quasi-hereditary algebras).

For example, we give some classes of full subcategories which satisfy our conditions.
We will see in Example \ref{Example:IQposet} more examples coming from combinatorics.

\begin{example}
  \leavevmode
  \begin{enumerate}
  \item For a hereditary algebra, the subcategory of projective (resp., injective) modules is
    closed under submodules (resp., quotients) of indecomposable modules.
  \item The class of all modules having simple socle (resp., simple top) is
    closed under submodules (resp., quotients) of indecomposable modules.
  \item The class of all thin modules is
    closed under both submodules and quotients of indecomposable modules.
    Here, a module $M$ is said to be \emph{thin} if
    the dimension $\dim_k Me$ is at most one for every primitive idempotent $e$.
  \item For a module $M$, we denote by $\Fac(M)$ the subcategory consisting of all modules which are quotients of direct sums of finite copies of $M$.
    Then, it is
    closed under quotients of indecomposable modules
    by definition.
  \item For a module $M$, the smallest full subcategory which contains $M$ and is
    closed under quotients of indecomposable modules
    is nothing but the additive closure of all indecomposable quotients of $M$.
    Notice that it differs from $\Fac(M)$ in general. \end{enumerate}
\end{example}

Now, we analyze right minimal $\mathcal{X}$-approximations for a given 
subcategory $\mathcal{X}$ which is closed under quotients of indecomposable modules. 

\begin{theorem} \label{thm:resolusion-quotient}
Let $\mathcal{X}$ be a full subcategory of $\mod A$ such that $\mathcal{X}=\add \mathcal{X}$ and is 
closed under quotients of indecomposable modules. 
For an $A$-module $M$, if it admits a right minimal $\mathcal{X}$-approximation $f\colon X\to M$, then the following statements hold. 
\begin{enumerate}[\rm (1)]
    \item The restriction of $f$ to each indecomposable direct summand of $X$ is injective.
    \item $\supp X \subseteq \supp M$. 
\end{enumerate}
Moreover, if $\mathcal{X}$ contains all projective $A$-modules, then $f$ is surjective and the equality holds in (2) of the statement.
\end{theorem}

\begin{proof}
(1) Let
\begin{equation}
    f = (f_i)_{i=1}^m \colon X=\bigoplus_{i=1}^m X_i \longrightarrow M
\end{equation}
be a right minimal $\mathcal{X}$-approximation of $M$, where $X_i\in \mathcal{X}$ are indecomposable.
For $j\in \{1,\ldots,m\}$, we have a decomposition 
\begin{equation}
    f_j \colon X_j \xrightarrow{f_j'} \Image f_j \xrightarrow{f_j''} M. 
\end{equation}
Since $\mathcal{X}$ is 
closed under quotients of indecomposable modules, $X_j\in \mathcal{X}$ implies $\Image f_j\in \mathcal{X}$. Since $f\colon X\to M$ is a right $\mathcal{X}$-approximation, 
there is a homomorphism $g_j\colon \Image f_j \to X$ satisfying $f_j'' = f\circ g_j$.
\begin{equation}\nonumber
\xymatrix{
    X_j \ar[rr]^-{f_j'} \ar[d]^{f_j} && \Image f_j \ar@{>}[dll]^-{f_j''} \ar@{.>}[d]^{g_j} \\
    M&& X. \ar[ll]^{f}
}
\end{equation}
In particular, we have
\begin{equation} \label{hom fgc}
f_j = f \circ g_j\circ f_j'.
\end{equation}

Now, let $f'$ and $g$ be the following homomorphisms:
\begin{equation}
\xymatrix{
    f' \colon X = {\displaystyle\bigoplus_{i=1}^m X_i}
    \ar[rr]^-{
    \begin{bsmallmatrix} f_1' & & \\ & \ddots & \\ & & f_m'
    \end{bsmallmatrix}
    }
    &&{\displaystyle\bigoplus_{i=1}^m X_i}
}
\quad \text{and} \quad
g\colon \xymatrix{
{\displaystyle\bigoplus_{i=1}^m X_i}
\ar[rr]^-{\begin{bsmallmatrix} g_1 & \hdots & g_m
\end{bsmallmatrix}} & & X.
}
\end{equation}
Then $h := g f'$ is an endomorphism of $X$ satisfying
$f=fh$
by \eqref{hom fgc}.
Since $f$ is right minimal, $h$ is an isomorphism. Since $h=g f'$, it implies that $f'$ is injective.
Furthermore, by the definition of $f'$, we have that $f_j'$ is injective for all $j\in \{1,\ldots,m\}$.
Therefore, $f_j$ is injective for all $j\in \{1,\ldots,m\}$.
This finishes the proof of (1).
The assertion (2) is immediate from (1).

Finally, we assume that $\mathcal{X}$ contains all projective modules. Then, $f$ is surjective. 
In particular, we have $\supp X \supseteq \supp M$. By (2), we obtain $\supp X = \supp M$ as desired.
\end{proof}

From now on, we fix a finite collection $\mathscr{X}$ of indecomposable $A$-modules and study $\mathscr{X}$-resolutions of modules over the support algebra. 
For an idempotent $e\in A$, let $B:=A/\langle 1-e \rangle$. 
We can regard $\mod B$ as a full subcategory of $\mod A$ via a natural surjection. 
In this case, $B$-modules are precisely $A$-modules whose supports are contained in $\mathcal{S}_e$, where $\mathcal{S}_e$ denotes the support of the semisimple $A$-module corresponding to $e$. 
In addition, for a given full subcategory $\mathcal{A}$ of $\mod A$, let $\bar{\mathcal{A}}:=\mathcal{A}\cap \mod B$. 
By the above discussion, $\bar{\mathcal{A}}$ consists of all modules in $\mathcal{A}$ whose supports are contained in $\mathcal{S}_e$. Especially, we have $\overline{\add \mathscr{X}} = \add \bar{\mathscr{X}}$ in our setting. 

\begin{lemma}\label{lem:barX}
If $\add \mathscr{X}$ is closed under quotients of indecomposable modules, then so is $\add \bar{\mathscr{X}}$. 
In addition, if moreover $\add \mathscr{X}$ contains all projective $A$-modules, then $\add \bar{\mathscr{X}}$ contains all projective $B$-modules. 
\end{lemma}

\begin{proof}
    We assume that $\add \mathscr{X}$ is closed under quotients of indecomposable modules. 
    The former assertion is clear since $\supp(N)\subseteq \supp(M)\subseteq \mathcal{S}_e$ holds for any indecomposable $B$-module $M$ and its quotient $N$. 
    On the other hand, every indecomposable projective $B$-module can be obtained by a quotient of some indecomposable projective $A$-module. 
    Thus, if moreover $\add \mathscr{X}$ contains all indecomposable projective $A$-modules, then 
    $\add \bar{\mathscr{X}}$ contains all indecomposable projective $B$-modules. 
    This shows the latter assertion. 
\end{proof}

Therefore, the notion of $\bar{\mathscr{X}}$-resolutions makes sense in the setting of Lemma \ref{lem:barX}.

\begin{proposition}\label{prop:resolution-supportalg}
    Assume that $\mathscr{X}$ is 
    closed under quotients of indecomposable modules
    and contains all indecomposable projective $A$-modules. Then, the following 
    hold for any $B$-module $M$.
    \begin{enumerate}[\rm (1)]
    \item A right minimal $(\add \bar{\mathscr{X}})$-approximation of $M_B$ is
      exactly a right minimal $(\add \mathscr{X})$-approximation of $M_A$.
    \item A right minimal $\bar{\mathscr{X}}$-resolution of $M_B$ is exactly a minimal right $\mathscr{X}$-resolution of $M_A$. 
    \item The $\bar{\mathscr{X}}$-resolution dimension of $M_B$  coincides with the $\mathscr{X}$-resolution dimension of $M_A$. 
    \end{enumerate}
    Therefore, if the $\mathscr{X}$-resolution global dimension of $A$ is finite, then we have 
    \begin{equation}\nonumber
        \bar{\mathscr{X}}\resgldim B \leq \mathscr{X}\resgldim A.
    \end{equation}
    In particular, $\bar{\mathscr{X}}$-resolution dimension of $B$ is also finite. 
\end{proposition}

\begin{proof}
    Let $M$ be an $B$-module. Regarding it as an $A$-module, we take a right minimal $(\add \mathscr{X})$-approximation $f\colon J\to M_A$ with $J\in \add \mathscr{X}$. By Theorem \ref{thm:resolusion-quotient}, we have $\supp(J)=\supp(M)\subset \mathcal{S}_e$. It means that $J\in \add \mathscr{X}\cap \mod B = \add \bar{\mathscr{X}}$. 
    Since $\mod B$ is a full subcategory of $\mod A$, $f$ gives a right minimal $(\add \bar{\mathscr{X}})$-approximation of $M_B$. 
    Since a right minimal approximation is determined up to isomorphism, we get the assertion (1). 
    Furthermore, by (1), we have $\Omega_{\mathscr{X}}(M_A) = \Omega_{\bar{\mathscr{X}}}(M_B)$.  
    Repeating this discussion, we obtain $\Omega^m_{\mathscr{X}}(M_A) = \Omega^m_{\bar{\mathscr{X}}}(M_B)$. 
    Thus, the assertion (2) and (3) hold.
    The last assertion follows from (3).
\end{proof}

The above proposition asserts that, for computing the $\mathscr{X}$-resolution of a given $A$-module $M$, we may reduce the algebra $A$ to its support algebra. More detail, it can be refined as follows. 

\begin{example}\label{Example:resolutiondimension}
    Assume that $\mathscr{X}$ is 
    closed under quotients of indecomposable modules
    and contains all indecomposable projective $A$-modules.
    In addition, we assume that the $\mathscr{X}$-resolution global dimension of $A$ is finite
    (for example, 
    if $\add \mathscr{X}$ is also
    closed under submodules of indecomposable modules, see Theorem \ref{resolution-finite}). 
    For a given $A$-module $M$, a right minimal $ \mathscr{X}$-resolution of $M$ can be computed in the following way. 
    \begin{enumerate}[\rm (1)]
        \item 
        Consider the support algebra of $M$, that is, $B_0 := A/\langle 1-\sum_{x\in \supp(M)}e_x \rangle$, and set $\mathscr{X}_0:=\mathscr{X}\cap \mod B_0$. 
        Then, we compute a right minimal $(\add \mathscr{X}_0)$-approximation $f \colon J \to M$ and a short exact sequence in $\mod B_0$: 
        \begin{equation*}
            0 \to K_1 \xrightarrow{\iota} J \xrightarrow{f} M \to 0. 
        \end{equation*}
        \item For an integer $i\geq 1$,  
        we regard $K_i$ as a module over its support algebra, say $B_i$, and set $\mathscr{X}_{i} := \mathscr{X}_{i-1} \cap \mod B_i$. 
        Then, we compute a right minimal $(\add \mathscr{X}_{i})$-approximation $f_i\colon J_i\to K_i$ and a short exact sequence in $\mod B_i$:
        \begin{equation*}
            0 \to K_{i+1} \xrightarrow{\iota_i} J_i \xrightarrow{f_i} K_i \to 0. 
        \end{equation*}
        \end{enumerate}
    Repeating (2), we finally get a chain complex
    \begin{equation}\label{res algorithm}
        \xymatrix@C=15pt{&\ar@{.}[r] & \ar[r]& J_2 \ar[rr]^{g_2}\ar@{.>}[rd]_{f_2}&&  J_1\ar[rr]^{g_1}\ar@{.>}[rd]_{f_1} && J \ar[rr]^{f} && M_A \ar[rr] && 0 \\
        &&&&\ar@{}[u]|{\circlearrowright}K_2\ar@{.>}[ur]_{\iota_1} && \ar@{}[u]|{\circlearrowright}K_1\ar@{.>}[ur]_{\iota}}
    \end{equation}
    and a filtration
    \begin{equation}\nonumber
        \cdots \subseteq \mathscr{X}_2 \subseteq  \mathscr{X}_{1} \subseteq \mathscr{X}_0 \subseteq \mathscr{X}. 
    \end{equation}
    Here, all modules and morphisms of \eqref{res algorithm} are considered in $\mod A$. 
    Then, $K_m$ is the $m$-th $\mathscr{X}$-syzygy of $M_A$ for any $m>0$ by Proposition \ref{prop:resolution-supportalg}.  
    If $0\neq K_m\in \add \mathscr{X}$ for some $m$, 
    then $K_i=0$ for all $i>m$. In this case, \eqref{res algorithm} gives a right minimal $\mathscr{X}$-resolution of $M_A$ and
    \begin{equation}\nonumber
        \mathscr{X}\resdim M_A = m.
    \end{equation}
\end{example}

\subsection{Application to convex full subposets}\label{sec:convex}
Let $P$ be a finite poset and $k[P]$ the incidence algebra of $P$ over a field $k$. Recall that $\mathscr{I}_P$ is the set of isomorphism classes of interval $k[P]$-modules, which contains all indecomposable projective modules, and also all indecomposable injective modules. 

\begin{proposition}[{\cite[Lemma 4.4 and its dual]{asashiba2023approximation}}]
\label{interval-subquot}
    The subcategory $\add \mathscr{I}_P$ is 
    closed under both submodules and quotients of indecomposable modules. 
\end{proposition}
It deduces the finiteness of the interval resolution global dimensions by using Theorem \ref{resolution-finite}.
More strongly, 
the following result is known.

\begin{theorem}[{\cite[Propositions~4.5,~4.9]{asashiba2023approximation}}] \label{intgldim-finite}
    For a finite poset $P$, the interval resolution global dimension is finite. Moreover, we have
    \begin{equation}\nonumber
        \intresgldim k[P] = 
        \max_{I\in \mathbb{I}(P)} \intresdim \tau (k_I) < \infty,
    \end{equation}
    where $\tau$ denotes the Auslander-Reiten translation for $\mod k[P]$.
\end{theorem}

We also mention that specific subclasses of intervals are of interest. 

\begin{example} \label{Example:IQposet}
\begin{enumerate}[\rm (1)]
    \item The subclass of all upper sets (resp., lower sets) are 
    closed under quotients (resp., submodule) of indecomposable modules.
    \item The subclass of all intervals having a single minimal element is
    closed under quotients of indecomposable modules.
    They are called \emph{single-source spreads} in \cite{blanchette2021homological}.
    Dually, a subclass provided by all intervals having single maximal element are 
    closed under submodules of indecomposable modules.
\end{enumerate}
\end{example}

Thanks to Proposition \ref{interval-subquot},
we can apply the results in the previous subsection to study interval resolutions of modules over $k[P]$.
The next result follows immediately from Theorem \ref{thm:resolusion-quotient}. 

\begin{corollary}\label{cor:interval cover}
    Let $P$ be a finite poset and $\mathscr{I}_{P}$ the set of isomorphism classes of interval modules. 
    For a given $k[P]$-module $M$, we take its interval cover $f\colon X=\bigoplus_{i=1}^m X_i \to M$, where all the $X_i$'s are interval modules. 
    Then, the following holds. 
    \begin{enumerate}[\rm (1)]
        \item $f$ is surjective. 
        \item $f|_{X_i}\colon X_i\to M$ is injective for every $i\in \{1,\ldots,m\}$. 
        \item $\supp X = \supp M$. 
    \end{enumerate}
     In particular, every $X_i$ can be taken as an interval submodule of $M$.
\end{corollary}

\begin{proof}
    Let $f\colon X \to M$ be an interval cover, that is, a right minimal $(\add \mathscr{I}_P)$-approximation of $M$. Then, it is surjective since $\mathscr{I}_{P}$ contains all indecomposable projective modules; see Subsection \ref{subsec:poset}.
    By Proposition \ref{interval-subquot}, we can apply Theorem \ref{thm:resolusion-quotient}
    with $\mathcal{X}=\add \mathscr{I}_P$ and obtain the assertions (2) and (3).
\end{proof}

\begin{remark}
    \label{rem:asashiba}
We note that Corollary~\ref{cor:interval cover} is essentially the same as \cite[Proposition 4.8]{asashibakoszul2023}. To see this,  
as a consequence of \cite[Proposition 4.8]{asashibakoszul2023}
we can deduce that 
$\begin{tikzcd}g : \bigoplus_{I \in S} X \rar{g_I} & M\end{tikzcd}$
is a right interval approximation of $M$, where $S$ is the set of all interval submodules of $M$, and $g_I$ are the corresponding inclusions (In fact, by \cite[Remark 4.9]{asashibakoszul2023}, a smaller set $S$ can be chosen). 
Then, some subset $S' \subseteq S$ gives 
the right minimal version (an interval cover)
$\begin{tikzcd}g':\bigoplus_{I \in S'} X \rar{g_I}& M\end{tikzcd}$
of $M$, from which our result follows. 
\end{remark}

In addition, we give an application to convex full subposets. 
Let $P'$ be a convex full subposet of $P$. In this situation, it is easy to see that the incidence algebra $k[P']$ of $P'$ can be written as $k[P']\cong k[P]/\langle 1- e\rangle$, where $e=\sum_{x\in P'} e_x$. 
Thus, we can regard $\mod k[P']$ as a full subcategory of $\mod k[P]$. 
In addition, since $P'$ is convex in $P$, every interval of $P'$ is an interval of $P$, that is, $\mathbb{I}(P')\subset \mathbb{I}(P)$. 
Furthermore, it is compatible with the map sending an interval to its associated interval module, in the sense that the following diagram commutes
\begin{equation}\nonumber
    \xymatrix{
    \mathbb{I}(P) \ar[r] \ar@{}[d]|{\bigcup} & 
    \mathscr{I}_P  \ar@{}[d]|{\bigcup} \\
    \mathbb{I}(P')  \ar[r] \ar@{}[ru]|{\circlearrowleft} & \mathscr{I}_{P'}
    }
\end{equation}

\begin{lemma}\label{lem:support}
    In the above, we have 
    \begin{equation}\nonumber
        \add \mathscr{I}_{P'} = \add \mathscr{I}_{P}\cap \mod k[P']. 
    \end{equation}
\end{lemma}

\begin{proof}
    Since $P'$ is convex, the set of intervals (resp., interval modules) of $P'$ is nothing but the set of intervals (interval modules) of $P$ whose support are contained in $P'$. Thus, we get the assertion. 
\end{proof}

Therefore, we obtain the following result. 

\begin{theorem}\label{thm:conv-resolution}
Let $P$ be a finite poset and $P'$ a convex full subposet of $P$. Then, the following statements hold for any $k[P']$-module $M$. 
\begin{enumerate}[\rm (1)]
\item A right minimal $(\add \mathscr{I}_{P'})$-approximation
  (interval cover over $P'$) of $M$
  is exactly a right minimal $(\add \mathscr{I}_{P})$-approximation
  (interval cover over $P$)
  of $M$.
\item A right minimal $\mathscr{I}_{P'}$-resolution
  (interval resolution over $P'$)
  of $M$ is exactly a minimal right $\mathscr{I}_{P}$-resolution
  (interval resolution over $P$)
  of $M$.
\item The interval resolution dimension of $M$ over $\mod k[P']$ is the same as that over $\mod k[P]$.
\end{enumerate}
\end{theorem}

\begin{proof}
    Suppose that $P'$ be a convex full subposet of $P$. 
    We have already seen that the incidence algebra $k[P']$ is of the form $k[P']\cong k[P]/\langle 1-e \rangle$, where $e = \sum_{x\in P'}e_x$. 
    By Lemma \ref{lem:support}, we have $\mathscr{I}_{P'} = \mathscr{I}_P \cap \mod k[P]$. In this situation, we can apply Proposition \ref{prop:resolution-supportalg} to $\mathscr{I}_{P'}$ and get the desired assertions (1)-(3). 
\end{proof}

\begin{corollary}\label{convex-mono}
    For any convex full subposet $P'$ of $P$, we have 
    \begin{equation}\nonumber
    \intresgldim k[P'] \leq
    \intresgldim k[\P].
    \end{equation}
\end{corollary}

\begin{proof}
    Recall that the interval resolution global dimension is finite (Theorem \ref{intgldim-finite}). 
    Then, the assertion is immediate from Theorem \ref{thm:conv-resolution}(3).
\end{proof}


\section{Monotonicity on interval resolution global dimension}\label{sec4}

In this section, we show the monotonicity of interval resolution global dimension of posets with respect to inclusion. 
Our main result is the following, which generalizes Corollary \ref{convex-mono} to not necessarily convex posets. 

\begin{theorem}\label{thm:monotone subposet}
    Let $P$ be a finite poset. For any (not necessarily convex) full subposet $P'$ of $P$, the following inequality holds.
    \begin{equation}\nonumber
        \intresgldim k[P'] \leq \intresgldim k[P]. 
    \end{equation}
\end{theorem}

Such monotonicity is interesting because it does not hold for the (usual) global dimension in general. 
In fact, we have the following example due to \cite{igusa1990cohomology}.

\begin{example}
  Let $P$ and $P'$ be posets given by
  \begin{equation*}
    P:
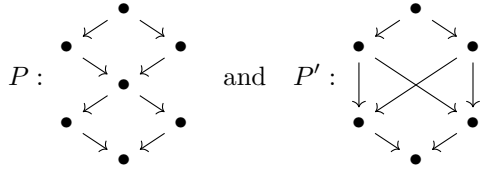
    \begin{tikzpicture}[baseline=0mm]
      \node(c) at(0,0) {$\bullet$};
      \node(t) at(0,1) {$\bullet$};
      \node(lt) at(-0.75,0.5) {$\bullet$};
      \node(rt) at(0.75,0.5) {$\bullet$};
      \node(b) at(0,-1) {$\bullet$};
      \node(lb) at(-0.75,-0.5) {$\bullet$};
      \node(rb) at(0.75,-0.5) {$\bullet$};
      \draw[->] (t)--(lt);
      \draw[->] (t)--(rt);
      \draw[->] (lt)--(c);
      \draw[->] (rt)--(c);
      \draw[<-] (b)--(lb);
      \draw[<-] (b)--(rb);
      \draw[<-] (lb)--(c);
      \draw[<-] (rb)--(c);
    \end{tikzpicture} \quad \text{and} \quad
    P':
    \begin{tikzpicture}[baseline=0mm]
      \node(t) at(0,1) {$\bullet$};
      \node(lt) at(-0.75,0.5) {$\bullet$};
      \node(rt) at(0.75,0.5) {$\bullet$};
      \node(b) at(0,-1) {$\bullet$};
      \node(lb) at(-0.75,-0.5) {$\bullet$};
      \node(rb) at(0.75,-0.5) {$\bullet$};
      \draw[->] (t)--(lt);
      \draw[->] (t)--(rt);
      \draw[->] (lt)--(lb);
      \draw[->] (lt)--(rb);
      \draw[->] (rt)--(lb);
      \draw[->] (rt)--(rb);
      \draw[<-] (b)--(lb);
      \draw[<-] (b)--(rb);
    \end{tikzpicture} 
  \end{equation*}
  respectively.
  Then, $P'$ is a full subposet of $P$, which is obtained by removing the point in the center.
  The global dimension of $k[P]$ is $2$ but that of $k[P']$ is $3$ (over an arbitrary field), see \cite[Section 3]{igusa1990cohomology}.
\end{example}

\begin{remark}
  \label{rem:nonfull}
  We note that if $P'$ is a subposet of $P$ but is not a full subposet,
  the inequality $\intresgldim k[P'] \leq \intresgldim k[P]$ may not hold.
  In fact, let $P'$ be any finite poset with $\intresgldim k[P'] > 0$.
  Then, it is well-known that there exists a total order on all the elements of $P'$
  compatible with the partial order of $P'$ (a linear extension of the partial order).
  Then, taking $P$ to be the elements of $P'$ together with this total order, $P'$ is a subposet of $P$.
  However, as we have seen in Example~\ref{Example:Glintdim=1}{\rm(1)}, $\intresgldim k[P] = 0$ since (the Hasse diagram of) $P$ is just the
  equioriented $A_n$-type quiver.
\end{remark}

\subsection{Results on resolution dimensions}\label{sec:sectionpair}
Let $A$, $B$ be finite dimensional $k$-algebras. Suppose that $\mathscr{X}$ (resp., $\mathscr{Y}$) is a finite collection of indecomposable $A$-modules (resp., $B$-modules) containing all indecomposable projective modules. In this situation, we may consider $\mathscr{X}\resgldim A$ and $\mathscr{Y}\resgldim B$ in the respective module categories.

We study resolution dimensions under the following assumption.

\begin{assumption}\label{niceFun}
  There are $k$-linear functors
  ${\rm F}\colon \mod A \to \mod B$ and ${\rm G}\colon \mod B \to \mod A$
  such that $\rm F$ is exact, there exists a natural equivalence
  $\Phi\colon 1_{\mod B} \overset{\sim}{\to} {\rm F}\circ {\rm G}$,
  and they restrict to functors between $\add \mathscr{X}$ and $\add \mathscr{Y}$, that is
\begin{equation}\nonumber
    \xymatrix{
    \mod A \ar@{->}@/_3mm/[rr]^{\rm F} \ar@{}[d]|{\rotatebox{0}{$\cup$}} & & \mod B \ar@{->}@/_3mm/[ll]_{\rm G} \ar@{}[d]|{\rotatebox{0}{$\cup$}}\\
    \add \mathscr{X} \ar@{->}@/_3mm/[rr]^{{\rm F}|_{\add \mathscr{X}}} & &\add \mathscr{Y} \ar@{->}@/_3mm/[ll]_{{\rm G}|_{\add \mathscr{Y}}}. 
    }
\end{equation}
\end{assumption}

We show the following.
\begin{theorem}\label{thm:monotone}
  Assume that the $\mathscr{X}$-resolution global dimension of $A$ is finite. If there exist functors ${\rm F},{\rm G}$ in Assumption \ref{niceFun}, then we have
  \begin{equation}\label{Y<X}
    \mathscr{Y}\resgldim B \leq \mathscr{X}\resgldim A.
  \end{equation}
  In particular, the $\mathscr{Y}$-resolution global dimension of $B$ is also finite.
\end{theorem}

To prove Theorem~\ref{thm:monotone},
we first consider pairs of modules satisfying the following condition.
\begin{condition}
  \label{conditionast}
  Under Assumption~\ref{niceFun},
  a pair $(X,M)\in \mod A\times \mod B$ satisfies this Condition
  if and only if
  there is a pair of morphisms $\alpha\colon {\rm G}(M)\to X$ and $\beta \colon {\rm F}(X) \to M$ such that
  \[
    \beta \circ {\rm F}(\alpha) \circ \Phi_M = 1_M.
  \]
\end{condition}
In particular, we observe that if Condition~\ref{conditionast} holds, then $\beta$ is a split epimorphism and thus
$M$ is a direct summand of ${\rm F}(X)$.

\begin{lemma}\label{sectionpair}
  Under Assumption~\ref{niceFun}, if the pair $(X,M)\in \mod A\times \mod B$ satisfies Condition~\ref{conditionast},
  then so does the pair $(\Omega^{m}_{\mathscr{X}}X, \Omega^{m}_{\mathscr{Y}} M) \in \mod A \times \mod B$ for any non-negative integer $m$.
\end{lemma}

\begin{proof}
  It suffices to show the case $m=1$. Suppose that $(X,M)$ satisfies Condition~\ref{conditionast}
  with the morphisms $\alpha, \beta$. 
  For a module $M\in \mod B$,
  let $f\colon I\to M$ be a right minimal $(\add \mathscr{Y})$-approximation of $M$ with $I\in \add \mathscr{Y}$.
  Similarly, let $g\colon J\to X$ be a right minimal $(\add \mathscr{X})$-approximation of $X$ with $J\in \add \mathscr{X}$.
  Then, the functor $\rm G$  gives a diagram
  \begin{equation}\label{ApplyG}
    \vcenter{
      \xymatrix{
        & {\rm G}(K) \ar[r]^-{{\rm G}(\imath)}  \ar@{.>}[d]^-{\alpha''} & {\rm G}(I) \ar[r]^-{{\rm G}(f)} \ar@{.>}[d]^-{\alpha'} & {\rm G}(M) \ar[d]^-{\alpha}\\ 
        0 \ar[r] & Z \ar[r]^-{\jmath} \ar@{}[ur]|\circlearrowleft& J \ar[r]^-{g} \ar@{}[ur]|\circlearrowleft& X \ar[r] & 0,  
      }}
  \end{equation}
  where $\imath\colon K:=\Omega_{\mathscr{Y}}(M) \to I$ and $\jmath \colon Z:=\Omega_{\mathscr{X}}(X) \to J$
  are kernels of morphisms $f$ and $g$ respectively, and where the bottom row is exact.
  Furthermore, since $g$ is a right $(\add \mathscr{X})$-approximation and ${\rm G}(I)\in \add \mathscr{X}$ by our assumption,
  there exists a morphism $\alpha'\colon {\rm G}(I)\to J$ such that $g\circ \alpha' = \alpha \circ {\rm G}(f)$.
  In addition, since $g \circ \alpha' \circ {\rm G} (\imath) = \alpha  \circ {\rm G}(f)\circ{\rm G}(\imath) = \alpha \circ{\rm G}(f\imath)=0$,
  there is a morphism $\alpha''\colon {\rm G}(K) \to Z$ such that $\jmath \circ \alpha'' = \alpha' \circ {\rm G}(\imath)$.

  Applying the exact functor ${\rm F}$ to \eqref{ApplyG}, we obtain the following diagram:
  \begin{equation}\label{ApplyF}
    \vcenter{
      \xymatrix@C=44pt{
        0 \ar[r] & K \ar[r]^-{\imath} \ar[d]^-{\Phi_K}& I \ar[r]^-{f} \ar[d]^-{\Phi_I}& M \ar[d]^-{\Phi_M} \ar[r] & 0 \quad \text{(exact)}\\
        0 \ar[r] & {\rm FG}(K) \ar[r]^-{\rm FG(\imath)} \ar[d]^-{{\rm F}(\alpha'')}\ar@{}[ur]|\circlearrowleft& {\rm FG}(I) \ar[r]^-{{\rm FG}(f)} \ar[d]^-{{\rm F}(\alpha')} \ar@{}[ur]|\circlearrowleft& {\rm FG}(M) \ar[d]^-{{\rm F}(\alpha)} \ar[r] & 0 \quad \text{(exact)}\\
        0 \ar[r] & {\rm F}(Z) \ar[r]^-{{\rm F}(\jmath)} \ar@{.>}[d]^-{\beta''} \ar@{}[ur]|\circlearrowleft& {\rm F}(J) \ar[r]^-{{\rm F}(g)} \ar@{.>}[d]^-{\beta'}\ar@{}[ur]|\circlearrowleft& {\rm F}(X) \ar[r] \ar[d]^-{\beta} & 0 \quad \text{(exact)} \\
        0 \ar[r] & K \ar[r]^-{\imath} \ar@{}[ur]|\circlearrowleft& I \ar[r]^-{f} \ar@{}[ur]|\circlearrowleft& M  \ar[r] & 0 \quad \text{(exact).}
      }}
  \end{equation}
  Since $f$ is a right $(\add \mathscr{Y})$-approximation and ${\rm F}(J)\in \add \mathscr{Y}$, there exists a morphism $\beta'\colon {\rm F}(J)\to I$ such that $f\circ \beta' = \beta \circ {\rm F}(g)$. In addition, it induces a morphism $\beta'' \colon {\rm F}(Z)\to K$ such that $\imath \circ \beta''=\beta' \circ{\rm F}(\jmath)$.
  
  By Condition~\ref{conditionast} and the commutativity of the boxes of the diagram \eqref{ApplyF}, we have
  \begin{equation}\nonumber
    1_M \circ f = (\beta\circ {\rm F}(\alpha)\circ \Phi_M) \circ f =
    f \circ (\beta' \circ {\rm F}(\alpha')\circ \Phi_I) = f\circ \varphi,
  \end{equation}
  where we set $\varphi := \beta'\circ {\rm F}(\alpha')\circ \Phi_I$ for simplicity.
  Since $f$ is right minimal, it implies that $\varphi$ is an isomorphism of $I$.
  Using $\varphi$, we can rewrite \eqref{ApplyF} as follows.
  \begin{equation}\nonumber
    \xymatrix@C=50pt{
      0 \ar[r] & K \ar[r]^-{\varphi \imath} \ar[d]_-{\beta''{\rm F}(\alpha'')\Phi_{K}} & I \ar[r]^{f=f\varphi^{-1}} \ar[d]^-{1_I}& M \ar[d]^-{1_M} \ar[r] & 0 \quad \text{(exact)} \\
      0 \ar[r] & K \ar[r]^-{\imath} \ar@{}[ur]|{\circlearrowleft}& I \ar[r]^-{f} \ar@{}[ur]|{\circlearrowleft}& M  \ar[r] & 0 \quad \text{(exact).}
    }  
  \end{equation}

  By the universal property of the kernel,  the map $\beta'' \circ  F(\alpha'')\circ\Phi_K$ must be an isomorphism, and hence there exists an isomorphism $\mu \colon K \to K$ such that  
  \begin{equation}\nonumber
    (\mu \circ \beta'') \circ {\rm F}(\alpha'')\circ \Phi_K = 1_K.  
  \end{equation}
  This shows that the pair $(Z=\Omega_{\mathscr{X}}(X),K = \Omega_{\mathscr{Y}}(M))$ satisfies Condition~\ref{conditionast}
  through the pair of morphisms $\alpha''\colon {\rm G}(K)\to Z$ and
  $\mu \circ \beta'' \colon {\rm F}(Z) \to K$. 

  This finishes the proof.
\end{proof}

Then, the following are immediate.
\begin{proposition}\label{prop:resdim}
  Under Assumption~\ref{niceFun}, suppose that the pair $(X,M)\in \mod A\times \mod B$
  satisfies Condition~\ref{conditionast}, and that
  the $\mathscr{X}$-resolution dimension of $X$ is finite.
  Then,
  \begin{equation}\nonumber
    \mathscr{Y}\resdim M \leq \mathscr{X}\resdim X.
  \end{equation}
\end{proposition}
\begin{proof}
  Applying Lemma \ref{sectionpair},
  the pair $(\Omega_{\mathscr{X}}^m X, \Omega_{\mathscr{Y}}^m M)$ also satisfies Condition~\ref{conditionast}
  for any non-negative integer $m$.
  Using the observation above,
  $\Omega_{\mathscr{Y}}^m M$ is a direct summand of ${\rm F}(\Omega_{\mathscr{X}}^m X)$.
  In particular, $\Omega_{\mathscr{X}}^m X = 0$
  implies $\Omega_{\mathscr{Y}}^m M = 0$, showing the claimed inequality.
\end{proof}

\begin{proof}[Proof of Theorem~\ref{thm:monotone}]
  For any given module $M\in \mod B$, the pair $({\rm G}(M),M)$ always satisfies Condition~\ref{conditionast}. 
  In fact, we just take $\alpha:=1_{G(M)}$ and $\beta:=\Phi_M^{-1}$ in the statement as
  \begin{equation}\nonumber
    \beta \circ {\rm F}(\alpha) \circ \Phi_M = \Phi_{M}^{-1} \circ 1_{\rm{FG}(M)} \circ \Phi_M = 1_M.
  \end{equation}
  Applying Proposition~\ref{prop:resdim}, we have
  \begin{equation}\nonumber
    \mathscr{Y}\resdim M \leq \mathscr{X}\resdim {\rm G}(M).
  \end{equation}
  Since $M$ is an arbitrary $B$-module, we get the desired equation \eqref{Y<X}.
\end{proof}

\subsection{Intermediate extension}
\label{sec:idempotent}
Let $A$ be a finite dimensional $k$-algebra. For a given idempotent $e\in A$,
we consider the idempotent subalgebra $B:=eAe$.
It is well-known that
the functors
\begin{equation}\nonumber
  \res_e(-) := (-)e, \
  \iind_e(-) := -\otimes_{B} eA, \
  \coind_e(-) := \Hom_{B}(Ae,-),
\end{equation}
respectively called the
restriction, induction, and coinduction functors,
provide a diagram
\begin{equation}\label{triple}
\begin{tikzcd}
    \mod A \ar[rr]{}{\res_e} & &
    \mod B \ar[ll,bend right,swap]{}{\iind_e} \ar[ll,bend left]{}{\coind_e}
  \end{tikzcd}
\end{equation}
with the following properties. See \cite[Chapter I.6]{ASS} or \cite[Section 4.1]{steinberg2016representation} for example
(We note however that \cite[Section 4.1]{steinberg2016representation} use left modules instead of right modules,
but the results hold as-is after taking this into account).

\begin{proposition}[{\cite[Theorem I.6.8]{ASS}\label{resTL}}]
  In the above setting, the following statements hold.
  \begin{enumerate}[\rm (a)]
  \item $\iind_e$ and $\coind_e$ are fully faithful functors such that
    $\res_e\circ \coind_e \cong 1_{\mod B} \cong \res_e\circ \iind_e$, and
    the functor $\coind_e$ is right adjoint to $\res_e$ and
    $\iind_e$ is left adjoint to $\res_e$;
    that is, we have natural isomorphisms
    \begin{eqnarray*}
      \Hom_A(X, \coind_e(M)) &\cong& \Hom_B(\res_e(X), M) \quad \text{and} \\
      \Hom_A(\iind_e(M),X) &\cong& \Hom_B(M, \res_e(X))
    \end{eqnarray*}
    for any $A$-module $X$ and $B$-module $M$.
    Thus, the diagram \eqref{triple} gives an adjoint triple.
  \item $\res_e$ is exact, $\iind_e$ is right exact, and $\coind_e$ is left exact.
  \item $\res_e$, $\iind_e$ and $\coind_e$ preserve indecomposability of modules.
    In addition, $\iind_e$ (resp., $\coind_e$)
    sends projective (resp., injective) modules to projective (resp., injective) modules.
  \end{enumerate}
\end{proposition}

  If we let ${\rm F} = \res_e$ and
  ${\rm G} = \iind_e$ (or ${\rm G} = \coind_e$),
  then indeed ${\rm F}$ is exact and there is a natural equivalence
  $\Phi\colon 1_{\mod B} \overset{\sim}{\to} {\rm F}\circ {\rm G}$
  by Proposition~\ref{resTL}.
  Of course, it is  too much to hope that this pair will restrict nicely to functors
  between arbitrary $\add \mathscr{X}$ and $\add \mathscr{Y}$ as in Assumption~\ref{niceFun}.
  In fact, in the setting of finite posets and their interval-decomposable modules,
  it is possible to construct examples where Assumption~\ref{niceFun} fails for this pair.
For example, we see below that the induction functor $\iind_e$ does not preserve the interval-decomposability of modules in general. 
\begin{example}
    Let $P$ and $P'$ be posets given by 
    \begin{equation}\nonumber
    P: \ 
    \begin{tikzpicture}[baseline=4mm]
    \node (a) at (0,1) {$1$};
    \node (b) at(-1,0) {$2$}; 
    \node (c) at(0,0) {$3$};
    \node (d) at(1,0) {$4$}; 
    \draw[<-] (a)--(c); 
    \draw[<-] (c)--(b); 
    \draw[<-] (c)--(d);
    \end{tikzpicture} \quad \text{and} \quad 
    P': \ 
    \begin{tikzpicture}[baseline=4mm]
    \node (a) at (0,1) {$1$};
    \node (b) at(-1,0) {$2$}; 
    \node (d) at(1,0) {$4$}; 
    \draw[<-] (a)--(b); 
    \draw[<-] (a)--(d);
    \end{tikzpicture}
\end{equation}
respectively. Then, $P'$ is a full subposet of $P$. We have an isomorphism $k[P'] \cong ek[P]e$ of $k$-algebras, where $e := 1-e_3$ is an idempotent of $k[P]$. 
Now, we regard $I:=P'$ as an interval of $P'$ and consider the corresponding interval $k[P']$-module $N:=k_{P'}$. 
In this case, the induced module $\iind_e N\in \mod k[P]$ is indecomposable and such that $\dim_k (\iind_e N)e_i = 1$ for $i\in \{1,2,4\}$ but $\dim_k (\iind_eN)e_3 = 2$.   
This shows that $\iind_e N$ is not an interval $k[P]$-module.
\end{example}

Thus, we consider the functor $\Theta:=\Theta_e : \mod B \rightarrow \mod A$ called the \emph{intermediate extension} \cite{kuhn1994generic} (the \emph{prolongement interm\'ediare} in \cite{BBD}), which is defined by using $\iind_e$ and $\coind_e$ in the following way.
For each $B$-module $M$, by the adjunctions, we get isomorphisms
\begin{eqnarray}
  \Hom_A(\iind_e(M), \coind_e(M)) \cong \Hom_B(\res_e(\iind_e M), M)  \cong
  \Hom_B(M,M) \quad \text{and} \\
  \Hom_A(\iind_e(M), \coind_e(M)) \cong \Hom_B(M, \res_e(\coind_e(M))) \cong \Hom_B(M,M)
\end{eqnarray}
Thus, the identity $1_M$ is associated to the map $\theta_M$ by
\begin{equation*}
    \xymatrix{
      \Hom_A(\iind_e(M), \coind_e(M)) \ar@{}[d]|{\rotatebox{90}{\text{$\in$}}} &
      \ar[l]_-{\sim} \Hom_B(M,M) \ar@{}[d]|{\rotatebox{90}{\text{$\in$}}} \\
    \theta_M   & \ar@{|->}[l] {1_M,}
    }
\end{equation*}
and an $A$-module 
\begin{equation}\nonumber
  \Theta(M) := \Image \theta_M \subseteq \coind_e(M).
\end{equation}
More precisely, $\theta_M$ is the map defined by $m\otimes ea \mapsto [xe\mapsto meaxe]$ for any $m\in M$ and $a,x\in A$.
  In fact, up to an isomorphism,
$\theta_M : \iind_e(M) \rightarrow \coind_e(M)$ is simply the couint $\epsilon: \iind_e \res_e \rightarrow 1_{\mod B}$
for the adjoint pair $(\iind_e, \res_e)$,
evaluated at $\coind_e M$, or the unit
$\eta: 1_{\mod A} \rightarrow \coind_e \res_e$ for the adjoint pair $(\res_e, \coind_e)$ evaluated
at $\iind_e M$.
That is, the following diagram commutes.
\begin{equation*}
  \begin{tikzcd}[column sep=4em]
    & \iind_e \res_e \coind_e M \ar[dr]{}{\epsilon_{\coind_e M}} & \\
    \iind_e M \ar[ur]{}{\cong} \ar[rr]{}{\theta_M} \ar[dr]{}{\eta_{\iind_e M}} & & \coind_e M \\
    & \coind_e \res_e \iind_e M \ar[ur]{}{\cong} &
  \end{tikzcd}
\end{equation*}

For a given morphism $f\colon M\to N$ of $B$-modules,
we have a commutative diagram
\begin{equation}\nonumber
  \xymatrix{
    \iind_e(M) \ar[d]^{\theta_M} \ar[rr]^{\iind_e(f)} && \iind_e(N) \ar[d]^{\theta_N} \\
    \coind_e(M) \ar[rr]^{\coind_e(f)} &\ar@{}[u]|\circlearrowleft& \coind_e(N)
  }
\end{equation}
since $\iind_e$, $\coind_e$ are bi-functorial.
Then, $\Theta(f)$ is defined to be the restriction of $\coind_e(f)$ to $\Theta(M)$:
\begin{equation}\nonumber
    \Theta(f) := \coind_e(f)|_{\Theta(M)}\colon \Theta(M) \longrightarrow \Theta(N).
\end{equation}
This gives a functor
\begin{equation}\label{fun-Theta}
    \Theta \colon \mod B \longrightarrow \mod A
\end{equation}
which is called the \emph{intermediate extension}. It is not an exact functor in general, but preserves monomorphisms and epimorphisms \cite[Proposition 4.6(4)]{kuhn1994generic}. 

Following \cite[Section 4.1]{steinberg2016representation}, 
we can provide the following equivalent construction for $\Theta$.
We additionally define the functors
$\trace_e : \mod A \rightarrow \mod A$
by $\trace_e(X) := XeA$ and
${\rm N}_e: \mod A \rightarrow \mod A/AeA$
by ${\rm N}_e(X) := \{x \in X \mid xAe = 0\}$.
Furthermore, there is the natural inclusion functor
$\iota_e : \mod A/AeA \rightarrow \mod A$, where a right $A/AeA$ module $V$ can be considered as a
right $A$ module via $va := v(a+AeA)$ for each $v \in V$ and $a \in A$.

Together with the above functors, we get what is known as a \emph{recollement} of abelian categories: 
\begin{equation}\label{diag:recollment}
  \begin{tikzcd}
    \mod A/AeA \ar[rr]{}{\iota_e} & &
    \mod A \ar[rr]{}{\res_e} \ar[ll, bend right,swap]{}{ (-)/\trace_e(-)} \ar[ll, bend left]{}{{\rm N}_e}& &
    \mod B . \ar[ll,bend right,swap]{}{\iind_e} \ar[ll,bend left]{}{\coind_e}
  \end{tikzcd}
\end{equation}

\begin{proposition}[{\cite[Corollary~4.12]{steinberg2016representation}}]
  Let $M \in \mod B$. Then,
  \begin{enumerate}[\rm (1)]
  \item $\Image \theta_M = \Theta(M) = \trace_e \coind_e M$
  \item $\Ker \theta_M = {\rm N}_e(\iind_{e}M)$
  \end{enumerate}
  Thus,
  \[
    \trace_e \coind_e M \cong \Theta(M) \cong \iind_e M / {\rm N}_{e}(\iind_e M).
  \]
\end{proposition}

\begin{proposition}\label{res-Theta}
    We have $\res_e \circ \Theta \cong 1_{\mod B}$. 
\end{proposition}

\begin{proof}
    We claim that $\res_e(\theta_M)$ provides the identity map $1_M$ for each $B$-module $M$. 
    In fact, there are isomorphisms 
    \begin{eqnarray*}
        &M\cong \res_e(\iind_e(M))& \quad  (m \mapsto (m\otimes e)e), \\
        &M\cong \res_e(\coind_e(M))& \quad (m\mapsto [xe \mapsto mexe]),
    \end{eqnarray*}
    and therefore
    \begin{equation}\nonumber
        M\cong \res_e(\iind_e(M)) \underset{\sim}{\xrightarrow{\res(\theta_M)}} \res_e (\coind_e(M)) \cong M
    \end{equation}
    is the identity map via these isomorphisms. 
    In addition, since $\res_e$ is exact, we find that 
    \begin{equation}\label{resTheta}
        \res_e(\Theta(M)) = \res_e (\Image \theta_M) = \Image(\res_e (\theta_M)) \cong M
    \end{equation} 
    by the above argument. Then, it clearly gives rise to a natural equivalence $\Phi\colon 1_{\mod B}\xrightarrow{\sim} \res_e\circ \Theta$ as desired. 
\end{proof}

We will use the following characterization of intermediate extensions. 

\begin{proposition}[{\cite[Proposition 4.6(3)]{kuhn1994generic}}] \label{prop:charTheta}
  Let $M$ be a $B$-module. For an $A$-module $X$, we have $\Theta(M)\cong X$ if and only if $X$ satisfies the conditions (i)-(iii) below:
    \begin{enumerate}[\rm (i)]
        \item $\res_e(X) \cong M$.
        \item For any proper submodule $Y$ of $X$, we have $\res_e(Y)\not \cong M$. 
        \item For any proper quotient $Z$ of $X$, we have $\res_e(Z)\not \cong M$. 
    \end{enumerate}
\end{proposition}

\subsection{Application to interval resolution dimensions}
In this section, we study the interval resolution global dimension over a given poset, and prove Theorem \ref{thm:monotone subposet}. 

Let $P$ be a finite poset and $k[P]$ the incidence algebra of $P$. Recall that every full subposet is completely determined by its elements. Thus, giving a set of pairwise orthogonal primitive idempotents of $k[P]$ is equivalent to giving a full subposet of $P$. Under this correspondence, we further obtain an isomorphism $ek[P]e \cong k[P']$ of $k$-algebras for any full subposet $P'\subseteq P$ and the corresponding idempotent $e:=\sum_{x\in P'} e_x$.

Now, we fix a full subposet $P'$ of $P$ and $e:=\sum_{x\in P'}e_x$ such that $ek[P]e\cong k[P']$ as in the previous paragraph. In this situation, there is an adjoint triple
\begin{equation}\label{tripleP}
\begin{tikzcd}
    \mod k[P] \ar[rr]{}{\res} & &
    \mod k[P'] \ar[ll,bend right,swap]{}{\iind} \ar[ll,bend left]{}{\coind}
  \end{tikzcd}
\end{equation}
given in Section \ref{sec:idempotent}, where we set $\res:=\res_e$, $\iind := \iind_e$ and $\coind := \coind_e$.
In addition, we obtain the functor $\Theta:=\Theta_e$ defined in \eqref{fun-Theta}.
We will check that $\res$ and $\Theta$ satisfy Assumption \ref{niceFun}
with respect to the interval-decomposable modules, which shows Theorem~\ref{thm:monotone subposet} by Theorem~\ref{thm:monotone}.

First, we verify the following for the restriction functor.

\begin{lemma}\label{res->int}
    $\res$ sends interval-decomposable modules to interval-decomposable modules.
\end{lemma}

\begin{proof}
    Since $P'$ is full, the intersection $J\cap P'$ is a convex set of $P'$ for any interval $J$ in $P$. Thus, we can decompose it into intervals as $J\cap P' = I_1\sqcup I_2 \sqcup \cdots \sqcup I_m$, where all $I_i$'s are intervals of $P'$. 
    Then, it is easy to see that $\res(k_J)\cong \bigoplus_{i=1}^m k_{I_i}$ as $k[P']$-modules.
\end{proof}

Next, we study the functor $\Theta$.
A key observation is that $\Theta$ realizes a combinatorial operation
called \emph{convex hull} in the following sense (see Proposition \ref{Theta->int}).
For a given subset $S\subseteq P$, 
the convex hull $\conv(S)$ of $S$ is defined to be the smallest convex full subposet of $P$ containing $S$.
We consider the map
\begin{equation}
    \conv \colon \mathbb{I}(P') \longrightarrow \mathbb{I}(P) 
\end{equation}
mapping $I\mapsto \conv(I)$.
While an interval $I$ of $P'$ is convex in $P'$ by definition,
  it is not automatically convex in $P$.
It can be checked that when $I$ is an interval of $P'$, then
\begin{equation}\label{IntMinMax}
    \conv(I)=\{x\in P\mid \text{$a\leq x \leq b$ for some $a,b\in I$}\} 
\end{equation}
forms an interval of $P$.
Note also that $\conv(I)\cap P' = I$ holds in this case.

\begin{proposition}\label{Theta->int}
    $\Theta$ sends interval modules to interval modules. More explicitly, for a given interval $I\in \mathbb{I}(P')$, we have $\Theta(k_I) \cong k_{\conv(I)}$. 
\end{proposition}

\begin{proof}
    Let $I$ be an interval of $P'$. 
    To prove the assertion, it suffices to check that $k_{\conv(I)}$ satisfies all properties (i)-(iii) in Proposition \ref{prop:charTheta}. 

    \begin{enumerate}[\rm (i)]
    \item From the proof of Lemma \ref{res->int}, one can easily check  $\res(k_{\conv(I)})\cong k_{\conv(I)\cap P'} = k_I$.
    
    \item We remind that every interval module has a $1$-dimensional vector space $k$ at each vertex lying in its support.
    For an interval $I$, we denote by $\min_{P'} (I)$ the set of minimal elements of $I$ in $P'$.
    By definition, it provides a set of minimal generators of the top of $k_{I}$. 
    Since $\min_{P'} I = \min_{P} \conv(I)$ holds by \eqref{IntMinMax}, it also provides a set of minimal generators of the top of $k_{\conv(I)}$. 
    Therefore, if we take a proper submodule $Y$ of $k_{\conv(I)}$, then there exists $a\in \min_{P'}I$ such that $a\not\in \supp(Y)$. In,  particular, we have 
    \begin{equation}\nonumber
        \supp(\res(Y)) = \supp(Y)\cap P' \subsetneq I = \supp(k_{I}). 
    \end{equation}
    So, $\res(Y)$ is not isomorphic to $k_{I}$.
    
  \item This is the dual of (ii).
  \end{enumerate}
  It finishes a proof.
\end{proof}

Consequently, we find that the functors $\res$ and $\Theta$ satisfy Assumption \ref{niceFun} with the diagram 
\begin{equation}\nonumber
    \xymatrix{
    \mod k[P] \ar@{->}@/_3mm/[rr]^{\rm \res} \ar@{}[d]|{\rotatebox{0}{$\cup$}} & & \mod k[P'] \ar@{->}@/_3mm/[ll]_{\Theta} \ar@{}[d]|{\rotatebox{0}{$\cup$}}\\
    \add \mathscr{I}_P \ar@{->}@/_3mm/[rr]^{\res|_{\add \mathscr{I}_P}} & &\add \mathscr{I}_{P'}, \ar@{->}@/_3mm/[ll]_{\Theta|_{\add \mathscr{I}_{P'}}}
    }
\end{equation}
where $\mathscr{I}_P$ (resp., $\mathscr{I}_{P'}$) is the set of isomorphism classes of interval modules over $P$ (resp., $P'$). In fact, $\res$ is exact (Proposition \ref{resTL}(b)) and $\res\circ \Theta \cong 1_{\mod k[P']}$ (Lemma \ref{res-Theta}). 
This gives a proof of our main result.

\begin{proof}[Proof of Theorem \ref{thm:monotone subposet}]
Under the above setting, we just apply Theorem \ref{thm:monotone} to $\res$ and $\Theta$ and get the assertion. 
\end{proof}


\section{A classification of posets with interval resolution global dimension zero}\label{Section:Dim=0}\label{sec5}

In this section, we give a complete classification of all finite posets with interval resolution global dimension zero (Theorem \ref{thm:gldim=0}). 
Now, for two positive integers $m,\ell>0$, let $C_{m,\ell}$ be a poset given by adding two distinguished points $\hat{0}$ and $\hat{1}$ to a disjoint union of $A_{m}(e)$ and $A_{\ell}(e)$ so that $\hat{0}$ (resp., $\hat{1}$) is its global minimum (resp., maximum). 
That is, the Hasse diagram of $C_{m,\ell}$ is given by the following quiver $Q_{m,\ell}$:
\begin{equation}\label{Pic:C_nm} 
    \vcenter{
    \xymatrix{
&& 1 \ar[r]^{\alpha_1} 
& \cdots \ar[r]^{\alpha_{m-1}}
& m \ar[rd]^{\alpha_m} \\ Q_{m,\ell} \colon &
\hat{0} \ar[ru]^{\alpha_0} \ar[rd]^{\beta_0}  & & & & \hat{1} \\
&& 1' \ar[r]^{\beta_1}  & \cdots \ar[r]^{\beta_{\ell-1}} & \ell' \ar[ru]^{\beta_{\ell}}
}}
\end{equation}

The aim of this section is to prove the following result. 

\begin{theorem}\label{thm:gldim=0}
Let $P$ be a poset with $n$ vertices and $k[P]$ the incidence algebra of $P$ over a field $k$. Then, the following conditions are equivalent. 
\begin{enumerate}[\rm (a)]
    \item Every $k[P]$-module is interval-decomposable. 
    \item Every indecomposable $k[P]$-module is interval. 
    \item $\intresgldim k[P] = 0$. 
    \item The Hasse diagram of $P$ is either $A_n(a)$ for some orientation $a$ or $C_{m,\ell}$ for some positive integers $m,\ell>0$ with $m+\ell = n-2$. 
\end{enumerate} 
In particular, these conditions do not depend on the characteristic of the base field $k$.
\end{theorem}

\subsection{Special biserial algebras}\label{subsec:SBA}
We recall definitions of special biserial algebras and string algebras. We refer to \cite{BR87,erdmann2006blocks} for basics of these algebras. 

Let $Q=(Q_0,Q_1)$ be a finite quiver, where $Q_0$ is the set of vertices of $Q$ and $Q_1$ is the set of arrows of $Q$. For arrows $\alpha$, we denote by $s(\alpha)$ and $t(\alpha)$ the starting point and the terminal point of $\alpha$ respectively. 
We denote by $\alpha\beta$ the path of length $2$ for two arrows $\alpha,\beta$ with $t(\alpha)=s(\beta)$.

\begin{definition}\label{def:SBA}
    Let $Q$ be a finite quiver and $I$ an ideal in the path algebra $kQ$ of $Q$. We say that $kQ/I$ is a \emph{special biserial algebra} if all the following conditions are satisfied: 
    \begin{enumerate}
        \item[\rm (SB1)] For each vertex $v$ in $Q$, there are at most two arrows starting at $v$, and there are at most two arrows ending at $v$. 
        \item[\rm (SB2)] For every arrow $\alpha$ in $Q$, there exists at most one arrow $\beta$ such that $t(\alpha)=s(\beta)$ and $\alpha\beta \not\in I$. 
        \item[\rm (SB3)] For every arrow $\alpha$ in $Q$, there exists at most one arrow $\gamma$ such that $s(\alpha)=t(\gamma)$ and $\gamma\alpha\not\in I$. 
    \end{enumerate}
    It is called \emph{string algebra} if in addition:
    \begin{enumerate}
        \item[\rm (SB4)] $I$ is generated by zero relations. 
    \end{enumerate}
\end{definition}

\begin{definition}
Let $Q=(Q_0,Q_1)$ be a finite quiver. 
For a given arrow $\beta$ in $Q$, we denote by $\beta^{-1}$ a formal inverse of $\beta$
and set $s(\beta^{-1})=t(\beta)$ and $t(\beta^{-1})=s(\beta)$.
We write $(\beta^{-1})^{-1} = \beta$.
The set of formal inverses of arrows in $Q_1$ is denoted by $Q_1^{-1}$. 
We say that a \emph{word} is a sequence $w=w_1w_2\cdots w_{\ell}$ where $w_j\in Q_1\cup Q_1^{-1}$ and $t(w_j) = s(w_{j+1})$ for all $j$. 
In this case, we set $s(w)=s(w_1)$, $t(w)= t(w_{\ell})$ and $w^{-1} = w_{\ell}^{-1}\cdots w_2^{-1}w_1^{-1}$.
A \emph{rotation} of a word $w=w_1w_2\cdots w_{\ell}$ with $s(w) = t(w)$ is a word of the form $w_{i+1}\cdots w_{\ell} w_1 \cdots w_i$.
\end{definition}

Now, we suppose that $kQ/J$ is a string algebra with  $(Q,J)$ satisfying (SB1)-(SB4). We say that
\begin{enumerate}[\rm (1)]
    \item a \emph{string} is a word $w=w_1w_2\cdots w_{\ell}$ in $Q$ such that $w_j\neq w_{j+1}^{-1}$ for all $j$ and no subword of $w$ or its inverse belongs to $J$. In addition, for each vertex $v\in Q_0$, let $1_v$ be the \emph{trivial string} at $v$
     We denote by ${\rm St}(Q,J)$ the set of representatives of strings under the relation $\sim$ which identifies each string with its inverse, 
    
   \item a \emph{band} is a non-trivial string $b$ in $Q$ such that $s(b)=t(b)$ and each power $b^m$ is a string,
     but $b$ itself is not a proper power of any string.
    We denote by ${\rm Ba}(Q,J)$ the set of representatives of bands under the relation $\sim'$ which identifies each band with its rotations and their inverses. 
\end{enumerate}

Next, we define string modules and band modules (we refer to \cite[Chapter II]{erdmann2006blocks} for the detail).
To each string $w \in {\rm St}(Q,J)$, we associate a \emph{string module} $M(w)\in \mod kQ/J$ as follows:
Let $w=w_1w_2\cdots w_m$ be a string. Consider the $A_{m+1}$-type quiver $Q_w$, where the arrows are labelled by $w_1,w_2,\ldots,w_m$ 
and we have $s(w_i) \to t(w_i)$ (resp., $s(w_i) \leftarrow t(w_i)$) if $w_i\in Q_1$ (resp., $w_i\in Q_i^{-1}$). Then, $M(w)$ is defined to be a module over $kQ/J$ obtained by replacing each vertex with the $1$-dimensional vector space $k$ and each arrow with the identity map. 
Notice that $M(w)\cong M(w')$ if and only if $w=w'$ or $w^{-1}=w'$. Each band $b\in {\rm Ba}(Q,J)$ defines a family of band modules, see \cite[Chapter II.3]{erdmann2006blocks}.
All string modules and band modules are indecomposable, and every indecomposable module over $kQ/J$ is either a string module or a band module \cite{BR87}.

Now, we consider special biserial algebras $A=kQ/I$. 
Let $\mathcal{L}$ be the set of vertices $v\in Q_0$ such that $e_{v}A$ is projective-injective and not uniserial, where $e_v$ denotes the corresponding idempotent at $v$. 
Then, $\bar{A}:=A/\bigoplus_{v\in \mathcal{L}}\soc (e_vA)$ is a string algebra. We have an embedding $\iota \colon \mod \bar{A} \to \mod A$ of module categories. 
Then, the set $\ind A$ of isomorphism classes of indecomposable $A$-modules is divided into three sub-classes: (i) projective-injective modules, (ii) string modules and (iii) band modules, where string/band modules are considered in $\mod \bar{A}$. 
In addition, it is known that $A$ is representation-finite if and only if there are no band modules. 

\subsection{Proof of Theorem \ref{thm:gldim=0}}
In this section, we prove Theorem \ref{thm:gldim=0}.
Firstly, we consider a poset $C_{m,\ell}$ for $m,\ell >0$.

\begin{proposition}\label{prop:strings}
Let $A$ be the incidence algebra of $C_{m,\ell}$. Then, it is special biserial. 
    Moreover, the interval module $k_{C_{m,\ell}}$ is the only indecomposable projective-injective, non-uniserial $A$-module up to isomorphism. 
\end{proposition}

\begin{proof}
Let $A$ be the incidence algebra of $C_{m,\ell}$. 
By definition, it is of the form $A=kQ_{m,\ell}/I$, where $Q_{m,\ell}$ is the quiver in \eqref{Pic:C_nm} and $I$ is a two-sided ideal generated by $\alpha_0\alpha_1\cdots\alpha_m - \beta_0\beta_1\cdots \beta_{\ell}$. 
Then, it is easy to check that the pair $(Q_{m,\ell},I)$ satisfies (SB1)-(SB3) in Definition \ref{def:SBA}. Thus, $A$ is special biserial. 
Since the poset $C_{m,\ell}$ has the maximum and minimum, there is a unique indecomposable projective-injective module given by $k_{C_{m,\ell}}$ up to isomorphism. In addition, it is non-uniserial in this case. Thus, we get the assertion. 
\end{proof}

Let $A=kQ_{m,\ell}/I$ be the incidence algebra of a poset $C_{m,\ell}$, where $Q_{m,\ell}$ is the quiver in \eqref{Pic:C_nm} and $I$ is a two-sided ideal generated by $\alpha_0\alpha_1\cdots\alpha_m - \beta_0\beta_1\cdots \beta_{\ell}$.   
By Proposition \ref{prop:strings}, it is special biserial and  
the corresponding string algebra is $\bar{A}=kQ_{m,\ell}/\bar{I}$, where $\bar{I}$
is generated by zero relations $\alpha_0\alpha_1\cdots\alpha_m$ and $ \beta_0\beta_1\cdots \beta_{\ell}$.
We will classify all indecomposable $A$-modules via the embedding $\iota \colon \mod \bar{A}\to \mod A$. 

\begin{proposition}\label{prop:st-ba}
    A complete set of representatives of ${\rm St}(Q_{m,\ell},\bar{I})$ is given by strings of the following form.
    \begin{enumerate}[\rm (i)]
        \item Trivial strings $1_v$ $(v\in Q_0)$; 
        \item $\alpha_i\alpha_{i+1}\cdots \alpha_j$ for $0\leq i \leq j \leq m$ and $(i,j)\neq (0,m)$; 
        \item $\beta_i\beta_{i+1}\cdots \beta_j$ for $0\leq i \leq j \leq \ell$ and $(i,j)\neq (0,\ell)$; 
        \item $(\beta_0\beta_1\cdots \beta_j)^{-1}(\alpha_0\alpha_1\cdots \alpha_i)$ for $0\leq i \leq m-1$ and $0\leq j \leq \ell-1$;
        \item $(\beta_j\beta_{j+1}\cdots \beta_\ell)(\alpha_i\alpha_{i+1}\cdots \alpha_m)^{-1}$ for $1\leq i \leq m$ and $1\leq j \leq \ell$;
        
    \end{enumerate}
    On the other hand, there are no bands in $(Q_{m,\ell}, \bar{I})$. 
\end{proposition}

\begin{proof}
    Recall that the ideal $\bar{I}$ is generated by paths $a_{0}a_1\cdots a_{m}$ and $b_0b_1\cdots b_{\ell}$. By a direct calculation, we can see that every string in $(Q_{m,\ell},\bar{I})$ is of the form one of (i)-(v). Moreover, we have no bands since every non-trivial string $w$ in (ii)-(v) satisfies $s(w)\neq t(w)$. 
\end{proof}

\begin{proposition}\label{prop:st-int}
    We have a bijection 
    \begin{equation} \label{bij:string-int}
    I\colon
       {\rm St}(Q_{m,\ell},\bar{I}) \longrightarrow \mathbb{I}(C_{m,\ell})\setminus \{C_{m,\ell}\},
    \end{equation}
    sending each string $w$ to an interval $I(w):=\supp(w)$,
    where $\supp(w)$ is the set of all vertices appearing as an endpoint of some arrow
    in $w$.
    Moreover, $\iota (M(w))\cong k_{I(w)}$ holds for any string $w$,
    where $\iota$ is the
      embedding $\iota \colon \mod \bar{A} \to \mod A$.
\end{proposition}

\begin{proof}
  From the description of strings in Proposition \ref{prop:st-ba},
  $I(w)$ is an interval of $C_{m,\ell}$ for any string $w$.
  Furthermore, since there are no repeated arrows in $w$,
  a string module $M(w)$ is just an interval module in $\mod A$ and is isomorphic to $k_{I(w)}$.
  On the other hand, every interval of $C_{m,\ell}$ except for $C_{m,\ell}$ itself can be written of this form.
  Thus, we get the assertions.
\end{proof}

\begin{corollary}\label{cor:Cml}
  Let $A$ be the incidence algebra of $C_{m,\ell}$.
Then, all indecomposable $A$-modules are interval. 
\end{corollary}

\begin{proof}
    According to a result explained in Section \ref{subsec:SBA}, every indecomposable $A$-module is isomorphic to a
    projective-injective module, a string module, or a band module.
    However, we have no band modules by Proposition \ref{prop:st-ba}. Thus,
    \begin{eqnarray*}
         \{\iota(M(w)) \mid w\in {\rm St}(Q_{m,\ell},\bar{I})\} \cup \{k_{C_{m,\ell}}\}
        &=& \{k_{I(w)} \mid w\in {\rm St}(Q_{m,\ell},\bar{I})\} \cup \{k_{C_{m,\ell}}\} \\ 
        &=& \{k_I \mid I\in \mathbb{I}(C_{m,\ell})\}
    \end{eqnarray*}
    forms the set of isomorphism classes of indecomposable $A$-modules,
    as desired. Here, we use Proposition \ref{prop:st-int} in the second equality.  
\end{proof}

\begin{corollary}
  Let $A$ be the incidence algebra of $C_{m,\ell}$. Then, the number of isomorphism classes of indecomposable $A$-modules is exactly
  \begin{equation}
   m\ell + \binom{m+\ell + 3}{2}.
  \end{equation}
\end{corollary}

\begin{proof}
We use the bijection \eqref{bij:string-int}. By a direct calculation, the number of the strings of ${\rm St}(Q_{m,\ell},\bar{I})$ for each case of (i)-(v) in Proposition \ref{prop:strings} is given by $a_{1} := m + \ell +2$, $a_{2}:=\binom{m+2}{2} -1 $, $ a_{3} :=\binom{\ell+2}{2} +1$, $a_{4}:=m \ell$ and $a_{5}:= m \ell$ respectively.
Then, we have  
\begin{equation*}
    \#\mathbb{I}(C_{m,\ell})
       = \# {\rm St}(Q_{m,\ell}, \bar{I}) + 1 
       = a_1 + a_2 +a_3 +  a_4 +a_5 + 1 
       = m\ell + \binom{m+\ell + 3}{2} 
\end{equation*}
as desired.
\end{proof}

Now, we are ready to prove Theorem \ref{thm:gldim=0}. 

\begin{proof}[Proof of Theorem \ref{thm:gldim=0}]
    The equivalences among (a), (b), and (c) are obvious. 
    In addition, we obtain (a) from (d) by combining the results 
    \begin{itemize}
        \item Gabriel's theorem for $A_n$-type quivers, and 
        \item Corollary \ref{cor:Cml} for posets $C_{m,\ell}$. 
    \end{itemize} 

    Thus, it suffices to show that (c) implies (d).
    Suppose that $P$ is a finite poset with interval resolution global dimension $0$. 
    In this case, $P$ has no element $v$ with degree greater than or equal to $3$ in its Hasse diagram. 
    In fact, if such an element $v$ exists, then $P$ must contain a full subposet whose Hasse diagram is the $D_4$-quiver $D_4(b)$ with some orientation $b$. We have already seen in Example \ref{Example:Glintdim=1}(2) that $D_4(b)$ has the interval resolution global dimension $1$. 
    Applying Theorem \ref{thm:monotone subposet}, we obtain  
    \begin{equation}\nonumber
            1 = \intresgldim k[D_4(b)] \leq 
            \intresgldim k[P] = 0, 
    \end{equation}
    a contradiction. 
    
    Therefore, we may assume that every element of $P$ has degree $\leq 2$. In this case, the Hasse diagram of $P$ is an acyclic quiver either of Dynkin type $A$ or extended Dynkin type $\tilde{A}$. 
    The former case is the desired one. 
    For the latter case, we need to exclude the cases when the Hasse diagram of $P$ has at least two sinks and/or at least two sources.
    In this situation, the corresponding incidence algebra is exactly a path algebra of type $\tilde{A}$ since there are no commutative relations.
    In particular, it is of tame representation type, but not finite representation type \cite{DR}. 
    So, there is an indecomposable module which is not interval, a contradiction. 
    Consequently, we get the assertion (d) as desired. This completes the proof.
\end{proof}



\section*{Acknowledgement}
We would like to thank Hideto Asashiba for
fruitful discussions.
We would also like to thank
Nicholas Kuhn for pointing out that the
functor $\Theta$ of subsection~\ref{sec:idempotent}
has already appeared in the literature.

This work is supported by JSPS Grant-in-Aid for Transformative Research Areas (A) (22H05105).
S.T. is supported by JST SPRING, Grant Number JPMJSP2148.


\noindent

\bibliographystyle{plain} 
\bibliography{IntervalCover.bib}

\end{document}